\documentclass[a4paper,reqno]{amsart}
\usepackage{ 
	amsmath, 
	amssymb, 
	amsthm, 
	amscd, 
	amstext, 
	verbatim, 
	enumerate, 
	mathtools,
	color}

\usepackage{hyperref}
\usepackage{comment}
\usepackage[all]{xy}
\usepackage[utf8]{inputenc} %
\usepackage{lmodern}        %
\usepackage[T1]{fontenc}    %
\usepackage{hyperref}       %
\usepackage{url}            %
\usepackage{booktabs}       %
\usepackage{amsfonts}       %
\usepackage{nicefrac}       %
\usepackage{microtype}      %
\usepackage{xcolor}         %
\usepackage{mathtools}		%
\usepackage{aliascnt} 		%
\usepackage{enumitem}		%
\usepackage{graphicx}       %
\usepackage{caption}        %
\usepackage{subcaption}     %
\usepackage{float} 			%
\usepackage{forest}			%
\usepackage{mathrsfs} 		%
\setlist[enumerate,1]{label={(\roman*)}} %
\makeatletter
\newcommand{\myitem}[1]{%
	\item[#1]\protected@edef\@currentlabel{#1}%
}
\makeatother

\usepackage{tikz}
\usetikzlibrary{arrows.meta}
\usepackage{tikz-cd}

\hypersetup{
	colorlinks,
	linkcolor={blue!80!black},
	citecolor={blue!80!black},
	urlcolor={blue!80!black}
}

\theoremstyle{plain}
\newtheorem{theorem}{Theorem}[section]

\newaliascnt{propCt}{theorem}
\newtheorem{prop}[propCt]{Proposition}
\aliascntresetthe{propCt}

\newaliascnt{lemmaCt}{theorem}
\newtheorem{lemma}[lemmaCt]{Lemma}
\aliascntresetthe{lemmaCt}

\newaliascnt{corCt}{theorem}
\newtheorem{cor}[corCt]{Corollary}
\aliascntresetthe{corCt}

\theoremstyle{definition}
\newaliascnt{defiCt}{theorem}
\newtheorem{defi}[defiCt]{Definition}
\aliascntresetthe{defiCt}

\newaliascnt{notaCt}{theorem}
\newtheorem{nota}[notaCt]{Notation}
\aliascntresetthe{notaCt}

\newaliascnt{remCt}{theorem}

\aliascntresetthe{remCt}

\newaliascnt{qstCt}{theorem}

\aliascntresetthe{qstCt}

\newaliascnt{pbmCt}{theorem}

\aliascntresetthe{pbmCt}

\newaliascnt{exaCt}{theorem}

\aliascntresetthe{exaCt}

\theoremstyle{plain}
\newtheorem{alphthm}{Theorem}			%

\newaliascnt{alphqstCt}{alphthm}

\aliascntresetthe{alphqstCt}

\newaliascnt{alphcorCt}{alphthm}
\newtheorem{alphcor}[alphcorCt]{Corollary}
\aliascntresetthe{alphcorCt}

\newaliascnt{alphpropCt}{alphthm}

\aliascntresetthe{alphpropCt}

\newaliascnt{alphproblemCt}{alphthm}

\aliascntresetthe{alphproblemCt}

\numberwithin{equation}{section}

\newcommand{\N}{\mathbb N}
\newcommand{\Z}{\mathbb Z}

\DeclareMathOperator{\Prob}{Prob}

\DeclareMathOperator{\id}{id} 
\DeclareMathOperator{\dom}{dom}

\begin{document}
\title{Dynamical comparison for local homeomorphisms}

\author[Shirly Geffen]{Shirly Geffen}
\address[Shirly Geffen]{Universit\"at M\"unster, Mathematisches Institut, Einsteinstr. 62, 48149 M\"unster, Germany}
\email{sgeffen@uni-muenster.de}
\urladdr{https://shirlygeffen.com/}

\author[Shanshan Hua]{Shanshan Hua}
\address[Shanshan Hua]{Universit\"at M\"unster, Mathematisches Institut, Einsteinstr. 62, 48149 M\"unster, Germany}
\email{shanshan.hua@uni-muenster.de}
\urladdr{https://sites.google.com/view/shanshan-hua/}

\author[Julian Kranz]{Julian Kranz}
\address[Julian Kranz]{Universit\"at M\"unster, Mathematisches Institut, Einsteinstr. 62, 48149 M\"unster, Germany}
\email{julian.kranz@uni-muenster.de}
\urladdr{https://sites.google.com/view/juliankranz/}


\begin{abstract}
We prove dynamical comparison for Deaconu--Renault groupoids associated to minimal surjective non-injective local homeomorphisms of compact metrizable spaces with finite Lebesgue covering dimension. 
As a corollary, the associated $C^*$-algebras are UCT Kirchberg algebras, recovering results by Carlsen--Thomsen via dynamical methods. 
In the zero-dimensional case, our result also verifies Matui's AH-conjecture for these groupoids using a recent breakthrough of Xin Li.
Our proof combines techniques from both the purely infinite and stably finite regimes:
We construct partial actions of non-abelian free groups as suitable ``large subgroupoids'' and establish comparison properties for these using the paradoxical towers technique developed by Gardella--Geffen--Kranz--Naryshkin. 
The boundary of the subgroupoid is controlled by a groupoid version of the topological small boundary property which we deduce from finite covering dimension of the unit space.
As a byproduct, we prove the classical small boundary property for minimal actions of countable discrete groups on finite-dimensional compact metrizable spaces without any freeness assumption.
\end{abstract} 

\maketitle
\tableofcontents

\section{Introduction} 
David Kerr introduced \emph{dynamical comparison} \cite{Kerr2020} for minimal group actions on compact Hausdorff spaces, with his 
original motivation being to identify dynamical properties of the action that imply regularity properties leading to classifiability of the associated crossed product $C^*$-algebra in the sense of the Elliott classification programme (see \cite{Winter2018,White2023} for surveys of the latter). 
In conjunction with the small boundary property \cite{Shub1991}, dynamical comparison is currently one of the most effective tools for establishing $\mathcal{Z}$-stability and thus Elliott-classifiability for crossed products associated to actions of amenable groups \cite{Kerr2020,KerrSzabo}.
Even beyond the classifiable setting, the combination of dynamical comparison and the uniform Rokhlin property \cite{Niu2022} produced powerful structural results including stable rank one in crossed products \cite{LiNiu,NaryshkinURPC,Bell25}.

Since its introduction, dynamical comparison has also led to many significant advances beyond operator algebras. 
In particular, it has recently appeared in connection with the Lindenstrauss--Weiss shift-embeddability problem \cite{Lindenstrauss2000,NaryshkinURPC} and Matui's AH-conjecture \cite{Matui12,Matui2016,Li2025}. 
The latter is originally stated in the setting of \'etale groupoids -- a commonly agreed upon framework for topological dynamical systems beyond group actions. 
Matui conjectures for any essentially principal minimal \'etale groupoid $G$ with Cantor unit space the existence of a short exact sequence 
\[H_0(G)\otimes \Z/2\to [[G]]_{\mathrm{ab}}\to H_1(G)\to 0,\]
where $[[G]]$ denotes the \emph{topological full group} of the groupoid. 
Topological full groups of \'etale groupoids are a fruitful class of groups, providing (counter-)examples for many open problems in group theory \cite{Juschenko2013,Juschenko2016,Nekrashevych2018,Skipper2019}.

Under the additional assumption of dynamical comparison (introduced and formulated in \cite{Ma2022} using open bisections in place of group elements), Li \cite{Li2025} proved a deep structural result relating the groupoid homology to the topological full group homology in terms of an underlying algebraic $K$-theory spectrum. 
His result generalizes famous homological stability and acyclicity results \cite{RandalWilliams2017,Szymik2019} and specializes to Matui's AH conjecture in low degrees. 
These developments make the problem of identifying \'etale groupoids with dynamical comparison both natural and increasingly central.

It is useful to distinguish between two rather different regimes -- the \emph{stably finite} regime and the \emph{purely infinite} regime: 
For dynamical systems admitting invariant probability measures, especially actions of amenable groups, dynamical comparison demands that whenever a closed subset of the space is strictly smaller than an open subset with respect to all invariant probability measures, then the former can be dissected into finitely many pieces and moved disjointly into the latter using the dynamics. 
In this case, the existence of invariant probability measures forces the associated $C^*$-algebra to be \emph{stably finite}.
On the other hand, for actions admitting no invariant probability measures, comparison reflects the paradoxical nature of the system: the measure-theoretic obstruction disappears, and comparison implies that the whole space can be dynamically compressed into any nonempty open subset. As a consequence, the associated $C^*$-algebras will typically be \emph{purely infinite}.

For actions of amenable groups, dynamical comparison has been verified in a number of important cases. 
Naryshkin proved that every minimal action of finitely generated groups with polynomial growth on a compact metrizable space has comparison \cite{Naryshkin22}; 
Glasner and Liu show in \cite{GL26} that every minimal free action of a countably infinite discrete amenable group on a finite-dimensional compact metrizable space has dynamical comparison, generalizing a sequence of results \cite{Downarowicz2023, KerrNaryshkin, NaryshkinPetrakos,NaryshkinEXT}. 
For actions of nonamenable groups, dynamical comparison has been established for many nonamenable groups under the assumption that the system is both minimal and amenable \cite{Gardella2023}. 
Moreover, every countable exact group admits a free minimal amenable action on the Cantor set with dynamical comparison \cite{Roerdam2012}.
Very recently, Boldrini--Prasad constructed the first examples of minimal topologically free actions of discrete groups which fail to have dynamical comparison \cite{Boldrini26}. 
Their examples are non-amenable and can be arranged both in the presence and in the absence of invariant probability measures.

In the setting of \'etale groupoids, much of the general theory surrounding dynamical comparison and almost finiteness has been developed \cite{AnantharamanDelaroche1997,Matui12,Rainone2018,Suzuki20,Boenicke2020,AraBonickeBosaLi,ABBL2020,Ma2022,Ma2026a}. 
However, compared to the case of group actions, relatively few ``unconditional'' results establishing dynamical comparison for a given class of groupoids exist. 
A notable contribution to the stably finite regime is Austad--B\"onicke's dynamical comparison result for groupoids of polynomial growth \cite{Austad2026}. 
The purpose of this paper is to identify a large class of groupoids with dynamical comparison in the purely infinite regime:
\begin{alphthm}[\autoref{thm-mainthm-dr}]
	Let $T\colon X\to X$ be a minimal surjective local homeomorphism of a compact metrizable space of finite covering dimension. Then the Deaconu--Renault groupoid $G_T$ satisfies dynamical comparison.
\end{alphthm}
For global homeomorphisms, the associated groupoids are stably finite and satisfy dynamical comparison by \cite{Downarowicz2023,Naryshkin22}. 
Our novel contribution is the case of non-injective local homeomorphisms.
The main corollary is the following consequence of Li's theorem \cite{Li2025}:
\begin{alphcor}
	Let $T\colon X\to X$ be a minimal surjective local homeomorphism of the Cantor set. Then the Deaconu--Renault groupoid $G_T$ satisfies Matui's AH-conjecture.
\end{alphcor}
In particular, this recovers the AH conjecture for graph groupoids \cite{Nyland2021}.
Our contribution is again to the case of non-injective local homeomorphisms since the case of global homeomorphisms is a classical result of Matui \cite{Matui2006}.

On the $C^*$-algebraic side, we deduce pure infiniteness of the groupoid. In this sense, our result can be understood as a dichotomy: Under natural assumptions, Deaconu--Renault groupoids are either stably finite or purely infinite, recovering a theorem by Carlsen--Thomsen \cite{Carlsen2012}, using dynamical tools:

\begin{alphcor}
	Let $T\colon X\to X$ be a minimal surjective non-injective local homeomorphism of a compact metrizable space of finite covering dimension. Then $C^*_r(G_T)$ is a UCT Kirchberg algebra.
\end{alphcor}
\begin{proof}
	Note that $G_T$ is topologically principal by \cite[Lemma~7.5]{Armstrong} and amenable by \cite[Lemma~3.5]{SimsWilliams}.
	Thus, $C^*_r(G_T)$ is nuclear by \cite{Renault1980}, simple and purely infinite by \cite[Corollary~1.2]{Ma2022} and satisfies the UCT by \cite{Tu1999}.
\end{proof}

The proof of our main theorem relies heavily on establishing dynamical comparison for certain minimal amenable \emph{partial actions} of nonabelian free groups with ``large domains'' (see \autoref{sec-partial}).
This part of the argument is inspired by earlier joint work of the first and third-named authors with Gardella and Naryshkin, where it was shown, among other things, that minimal amenable global actions of non-elementary hyperbolic groups have dynamical comparison \cite{Gardella2023}.
The precise link to Deaconu--Renault groupoids is given by a recent theorem of de Castro and Steinberg \cite[Theorem~2.2]{Steinberg2026} which realizes every Deaconu--Renault groupoid with totally disconnected unit space as the transformation groupoid of a \emph{partial group action} $F_d\curvearrowright X$ given by homeomorphisms $\{\alpha_g\colon D_{g^{-1}}\xrightarrow{\cong} D_g\}_{g\in F_d}$ between open subsets of $X$. 
Our proof of dynamical comparison in the totally disconnected case relies on the additional assumption that at least two of the generators $a_1,\dotsc,a_d\in F_d$ act via a surjection $\alpha_g\colon D_{g^{-1}}\xrightarrow{\cong} X$. 
Equivalently, the local homeomorphism $T\colon X\to X$ has fibers of cardinality at least two. 
The reduction to this special case is the content of \autoref{sec-loc-homeo}. 

In higher dimensions, de Castro--Steinberg's realization is no longer directly available. Instead, we mimic the zero-dimensional construction by building a ``large open subgroupoid'' of $G_T$ arising from a partial action of a non-abelian free group, and prove a relative comparison result for this subgroupoid.
The main technical difficulty of the higher-dimensional case is then to control the boundary of this subgroupoid and prove that it can be ignored for purposes of dynamical comparison.
This part of the argument is based on the finite-dimensional general-position methods developed by Lindenstrauss \cite{Lindenstrauss} and Szab{\'o} \cite{Szabo15}, together with the small-boundary techniques of Kerr--Szab{\'o} \cite{KerrSzabo}. 
To this end, we define a formally weaker variant of Szab\'o's \emph{topological small boundary property} which we call the \emph{thin boundary property} (see \autoref{def-thin}) in analogy to Buck's original definition for $\Z$-actions \cite{Buck}.
We extend the aforementioned methods to the groupoid setting and prove the following theorem:
\begin{alphthm}[\autoref{cor-sbp-groupoid}]
	Let $G$ be a second countable Hausdorff minimal \'etale groupoid with a compact metrizable unit space of finite covering dimension. Then $G$ satisfies the thin boundary property. 
\end{alphthm}
In contrast to the results in \cite{Szabo15,Gardella2024}, our result does not incorporate any freeness assumptions. 
As a corollary, we obtain the following result about the classical small boundary property \cite{Shub1991} for group actions:
\begin{alphcor}[\autoref{cor-sbp-group}]
	Let $\Gamma\curvearrowright X$ be a minimal action of a countable discrete group on a compact metrizable space of finite covering dimension. Then $\Gamma\curvearrowright X$ satisfies the small boundary property. 
\end{alphcor}

Small-boundary methods have so far been used primarily in the stably finite setting, where they play a central role in proving finite nuclear dimension and $\mathcal Z$-stability for crossed products associated to amenable group actions \cite{Szabo15,Elliott2017,KerrSzabo}.
From this point of view, the finite-dimensional part of our argument shows that small-boundary techniques are also useful in the purely infinite regime. This is somewhat unexpected: previous results establishing dynamical comparison for minimal amenable actions of nonamenable groups, such as those in~\cite{Gardella2023}, did not require any dimensionality assumptions on the underlying space. It remains an interesting question whether the results presented in this work generalize to arbitrary compact metrizable spaces.

In the zero-dimensional case, combinatorial graph-based models for surjective local homeomorphisms have been developed \cite{AraExel,AraClaramunt}. 
In light of the available literature on pure infiniteness for (generalized) graph $C^*$-algebras (see \cite{Bates2000,Hong2003,Sims2006,Pask2006,Ara2012,Pask2017,Pask2021,Pask2026} for an incomplete list), it would be interesting to know to what extent our main result can be obtained from these models.

\subsection*{Acknowledgments} 
The authors were supported by the Deutsche Forschungsgemeinschaft
(DFG, German Research Foundation) under Germany's Excellence Strategy EXC 2044/2-390685587,
Mathematics M{\"u}nster: Dynamics--Geometry--Structure, and 
by the SFB 1442 of the DFG. 
The first and third-named authors gratefully acknowledge Eusebio Gardella and Petr Naryshkin for discussions during a visit to KU Leuven in 2022, where some of the 
ideas concerning dynamical comparison for minimal amenable partial actions of free groups were developed. 

\subsection*{AI statement} 
All of the main ideas of this paper, including the research question, the overall proof strategy and the proofs themselves have been developed by the authors with the exception of preliminary drafts of the proofs of Lemma 5.2 and Theorem 5.8 which were developed with the help of ChatGPT Plus 5.5, closely following \cite{Szabo15}, and heavily revised by the authors, as well as two minor technical gaps in the proofs of Lemmas 3.2 and 4.2 which have been fixed using helpful comments by ChatGPT Plus 5.5.
The authors take full responsibility for the content of this paper.

\section{Dynamical comparison and thin boundaries}\label{sec-dyn-comp}
This section reviews the basic definitions relevant to the paper, including dynamical comparison for groupoids. 
We also define the \emph{thin boundary property} for groupoids which is inspired by the \emph{topological small boundary property} from topological dynamics \cite{Lindenstrauss,Szabo15,KerrSzabo} and borrows its name from \cite{Buck}. 

A \emph{groupoid} is a small category all of whose morphisms are invertible. 
We identify a groupoid $G$ with its set of all morphisms. 
The \emph{unit space} $G^{(0)}\subset G$ is the set of all identity morphisms, which we identify with the set of objects.
We denote by
\[G^{(2)}\coloneqq \{(g,h)\mid s(g)=r(h)\}\subset G\times G\]
the set of \emph{composable pairs}. 
Equivalently, a groupoid is encoded by a set $G$ and a subset $G^{(0)}\subset G$ as above together with \emph{range}, \emph{source}, \emph{composition}, and (uniquely determined) \emph{inversion} maps
\[r,s\colon G\to G^{(0)},\quad -\cdot -\colon G^{(2)}\to G,\quad (-)^{-1}\colon G\to G\]
satisfying the usual category axioms.  

A \emph{topological groupoid} is a groupoid $G$ equipped with a topology such that the range, source, composition and inversion maps are continuous. 
A locally compact groupoid $G$ is called \emph{\'etale} if the range and source maps are local homeomorphisms, i.e. every point $g\in G$ has an open neighbourhood $g\in B\subset G$ such that $r|_B\colon B\to r(B)$ and $s|_B\colon B\to s(B)$ are homeomorphisms onto their images. 
Such a set $B$ is called an \emph{open bisection}. 
The inverse semigroup of open bisections naturally acts on the unit space via homeomorphisms between open subsets of $G^{(0)}$:
\begin{nota}\label{notation-theta}
For an \'etale groupoid $G$ and an open bisection $B\subset G$, we denote by $\theta_B\colon s(B)\xrightarrow{\cong} r(B)$ the unique homeomorphism defined by $\{\theta_B(x)\}=r(s^{-1}(x)\cap B)$. 
For a subset $A\subset G^{(0)}$, we also write $\theta_B(A)\coloneqq \theta_B(A\cap s(B))$. 
\end{nota} 

A subset $U\subset G^{(0)}$ is called \emph{invariant} if it satisfies $U=r(s^{-1}(U))$. 
We say that $G$ is \emph{minimal} if $\emptyset$ and $G^{(0)}$ are the only closed invariant subsets of $G^{(0)}$. 

We recall the definition of dynamical comparison for \'etale groupoids. 
The underlying idea had previously appeared in lectures by Winter; see \cite[Introduction]{Kerr2020}. 
The notion was formulated for minimal $\mathbb{Z}$-actions subsequently by Buck \cite{Buck}, 
and further developed for general group actions by Kerr \cite{Kerr2020}. 
It is the dynamical analogue of 
$C^*$-algebraic \emph{strict comparison of positive elements} in the context of the Toms--Winter conjecture \cite{Toms2009}. 
The definition for \'etale groupoids first appeared in \cite{Ma2022} under the name \emph{groupoid comparison}. 

\begin{defi}[Dynamical comparison]
    Let $G$ be a Hausdorff \'etale groupoid. 
    \begin{enumerate}
        \item A closed subset $A\subset G^{(0)}$ is said to be \emph{dynamically below} an open subset $U\subset G^{(0)}$, written $A\prec_G U$ (or $A\prec U$ if $G$ is understood), if there are open bisections $B_1,\dotsc,B_n\subset G$ such that $A\subset s(B_1)\cup \dotsb \cup s(B_n)$ and $r(B_1),\dotsc,r(B_n)$ are pairwise disjoint subsets of $U$. 
        \item A \emph{$G$-invariant measure} on $G^{(0)}$ is a Borel probability measure $\mu$ on $G^{(0)}$ such that $\mu(s(B))=\mu(r(B))$ for all open bisections $B\subset G$. 
        \item Suppose furthermore that $G$ is minimal.\footnote{We restrict here to minimal groupoids, since in the non-minimal setting, the appropriate formulation of dynamical comparison is more subtle and must take the invariant ideal structure into consideration; see \cite[Definition 4.2]{Austad2026} and \cite[Definition 2.1 and Lemma 2.4]{AraBonickeBosaLi} for relevant discussions. }  
        Then $G$ is said to satisfy \emph{dynamical comparison} if for every closed set $A\subset G^{(0)}$ and every non-empty open set $U\subset G^{(0)}$ satisfying $\mu(A)<\mu(U)$ for all $G$-invariant Borel probability measures $\mu$ on $G^{(0)}$, we have $A\prec U$.\footnote{Our formulation differs slightly from that of \cite[Definition 3.4]{Ma2022}, which is stated for two open subsets. Under our standing assumption that $G^{(0)}$ is compact, the two notions are equivalent. }  
        If, in addition, there are no $G$-invariant Borel probability measures on $G^{(0)}$, then we say that $G$ satisfies \emph{dynamical comparison with no invariant measures}. 
    \end{enumerate} 
\end{defi} 

To establish dynamical comparison for Deaconu--Renault groupoids with higher-dimensional unit space, we introduce a groupoid regularity property called the \emph{thin boundary property}. 
It allows us to replace clopen partitions of zero-dimensional spaces with open partitions whose boundaries are negligible for dynamical comparison. 
The definition is inspired by the \emph{small boundary property} for amenable group actions, introduced by Shub--Weiss \cite{Shub1991} and coined by Lindenstrauss \cite{Lindenstrauss}, and by the stronger \emph{topological small boundary property} \cite{Buck,Szabo15,KerrSzabo}.
For transformation groupoids of $\mathbb{Z}$-actions, our notion is stronger than the former, and weaker than the latter. 
We therefore adopt different terminology, following Buck's notion of thin sets for $\mathbb{Z}$-actions \cite[Definition 3.4]{Buck}.

\begin{defi}[Thin boundaries]\label{def-thin} 
    Let $G$ be a Hausdorff \'etale groupoid. 
    \begin{enumerate} 
    \item A closed subset $A\subset G^{(0)}$ is called \emph{$G$-thin}, if $A\prec U$ for every non-empty open subset $U\subset G^{(0)}$. 
    \item We say that $G$ has the \emph{thin boundary property} if there is a basis $\mathcal B$ for the topology of $G^{(0)}$ such that for each $U\in \mathcal B$, the boundary $\partial U$ is $G$-thin. 
    \end{enumerate} 
\end{defi} 

Note that every Hausdorff \'etale groupoid with totally disconnected unit space has the thin boundary property since its unit space has a basis consisting of clopen subsets.
The 
thin boundary property holds more generally for finite-dimensional unit spaces:
\begin{theorem}\label{thm-Lebesgue-thin}
    Let $G$ be a minimal second countable Hausdorff \'etale groupoid with a compact unit space $G^{(0)}$ of finite Lebesgue covering dimension. Then $G$ has the thin boundary property.
\end{theorem}

We postpone the proof of \autoref{thm-Lebesgue-thin} to \autoref{sec-dimension} and first derive the proof of our main results assuming \autoref{thm-Lebesgue-thin}.
Readers who are only interested in the case of zero-dimensional unit spaces can safely ignore \autoref{sec-dimension} and replace ``thin boundaries'' by ``empty boundaries'' for the remainder of this paper. 

\section{Local homeomorphisms with large fibers}\label{sec-loc-homeo}
Deaconu--Renault groupoids associated to local homeomorphisms were introduced and studied in \cite{Renault1980,Deaconu1995,AnantharamanDelaroche1997}. 
This section recalls the basic properties of Deaconu--Renault groupoids and reduces the main theorem to the special case of minimal surjective local homeomorphisms with ``large fibers'' in the sense that the preimage of any point has cardinality at least two. 
This special case is then treated in \autoref{sec-partial} assuming in addition the thin boundary property (\autoref{def-thin}). The latter assumption is reduced to finite covering dimension of the unit space in \autoref{sec-dimension}.

\begin{defi} 
Let $X$ be a locally compact Hausdorff space, let $\dom(T)\subset X$ be an open subset and let $T\colon \dom(T)\to X$ be a local homeomorphism. 
For an integer $n\geq 0$, we recursively define the \emph{domain} of $T^{n}$ as follows: $\dom(T^0)=X$ and 
\[\dom(T^{n+1})=\{x\in \dom(T^n)\mid T^n(x)\in \dom(T)\}.\]
The \emph{Deaconu--Renault groupoid} $G_T\subset X\times \Z\times X$ associated to $T$ is given by 
\[G_T\coloneqq \{(x,n-m,y)\mid  n,m\geq 0,x\in \dom(T^n),y\in \dom(T^m),T^n(x)=T^m(y)\},\]
with range, source, and multiplication maps given by 
\[r(x,n,y)=x,\quad s(x,n,y)=y,\quad (x,n,y)(y,m,z)=(x,n+m,z).\]
A basis of the topology of $G_T$ is given by the basic open sets 
\begin{equation}\label{eq-basis-dr}
    B_T(U,V,n,m)\coloneqq \{(x,n-m,y)\mid x\in U,y\in V, T^n(x)=T^m(y)\},
\end{equation}
where $n,m\geq 0$ are integers and $U\subset \dom(T^n),V\subset \dom(T^m)$ are open subsets such that $T^n|_U$ and $T^m|_V$ are injective. 
\end{defi}

A local homeomorphism $T\colon X\to X$ is called \emph{minimal} if the associated Deaconu--Renault groupoid $G_T$ is minimal. 
In this case, it means that there is no non-empty proper closed subset $Y\subset X$ such that 
\begin{equation} 
\label{equ:minimal_equi} 
Y = \{x\in X: T^{n}(x) = T^{m}(y) \text{ for some }y\in Y, n, m \geq 0\}. 
\end{equation} 
The condition \eqref{equ:minimal_equi} is equivalent to 
\[
T(Y)\subset Y
\qquad\text{and}\qquad
T^{-1}(Y)\subset Y.
\]
If $T$ is surjective, the above conditions are equivalent to $T(Y)=Y=T^{-1}(Y)$.\footnote{Note that $Y = T^{-1}(Y)$ holds without surjectivity: if $T(Y)\subseteq Y$, then $Y\subset T^{-1}(Y)$, which combines with $T^{-1}(Y)\subseteq Y$ and gives the equality. When $T$ is surjective, for any $y\in Y$, there exists $x$ such that $T(x) = y$. Hence, $x\in T^{-1}(Y)\subset Y$, which implies $Y\subset T(Y)$. } 

The following key lemma states that, up to finite iteration, every surjective minimal local homeomorphism is either a homeomorphism or has ``large fibers''. 
\begin{lemma}\label{lem:unif_non_inj} 
	Let $T\colon X\to X$ be a minimal, surjective and non-injective local homeomorphism of a compact Hausdorff space. 
    Then there is an integer $n\geq 1$ such that for every $x\in X$, the set $T^{-n}(x)$ has cardinality at least $2$. 
\end{lemma} 

\begin{proof}
	Since $T$ is a non-injective local homeomorphism, the set 
	\[L_n\coloneqq \{x\in X\mid |T^{-n}(x)|\geq 2\}\subset X\]
	is nonempty and open for every $n\geq 1$.  
    Surjectivity of $T$ implies that $L_n\subset L_{n+1}$, for all $n\geq 1$.
	By compactness it therefore suffices to show that 
    \[L\coloneqq \bigcup_{n\geq 1} L_n=X.\] 
	First, note that $T(L_n)\subset L_{n+1}$ for all $n\geq 1$, so that $T(L)\subset L$ and therefore also $L\subset T^{-1}(L)$. 
	Hence, the union 
	\[\widetilde L\coloneqq \bigcup_{n\geq 1} T^{-n}(L)\]
    is an increasing union.
	By definition, $T^{-1}\big(\widetilde L\big)= \bigcup_{n\geq 2} T^{-n}(L)\subset \widetilde L$.
	We get 
	\begin{align*}
		T\big(\widetilde L\big)= \bigcup_{n\geq 1}T(T^{-n}(L))\subset L\cup \widetilde L\subset T^{-1}(L)\cup \widetilde L=\widetilde L.
	\end{align*}
	Since $\widetilde L$ is open, invariant, and non-empty due to surjectivity and non-injectivity of $T$, we conclude from minimality that $\widetilde L=X$. 
	By compactness, $X=T^{-n}(L)$ for some $n\geq 1$.
    By surjectivity of $T$, we conclude that $X=T^n(X)=L$. 
\end{proof} 

As shown above, passing to an iteration $T^{n}$ provides a uniform lower bound on the fibre cardinalities, but we might lose minimality. 
The next lemma resolves this issue by restricting $T^{n}$ to a clopen subset on which it is minimal, and thus producing a minimal open Deaconu--Renault subgroupoid. 

\begin{lemma}\label{lem:min_n_set}
	Let $T\colon X\to X$ be a minimal surjective local homeomorphism of a compact Hausdorff space. 
	Let $n\geq 1$ be an integer.
    Then there is a non-empty clopen set $A\subset X$ 
    such that $A=T^n(A)=T^{-n}(A)$ and such that the restricted local homeomorphism $T^n|_A\colon A\to A$ is surjective and minimal.
\end{lemma}
\begin{proof}
	By Zorn's lemma, there is a non-empty closed set $A\subset X$ which is minimal, with respect to inclusion, among all closed subsets of $X$ invariant under the local homeomorphism $T^n\colon X\to X$. 
    Indeed, consider the family 
    \[\mathcal{F}\coloneqq \{ B\subset X\mid B \text{ is closed and non-empty, } T^n(B)\subset B, \text{ and } T^{-n}(B)\subset B\}, \]
    ordered by reverse inclusion. 
    Since $X\in \mathcal{F}$, this family is non-empty. 
    Intersections of chains are non-empty by compactness and remain in $\mathcal{F}$, so every chain has an upper bound. 
    It follows from Zorn's lemma that $\mathcal{F}$ admits a maximal element $A\in\mathcal{F}$. 
    Since $T^{n}: X \rightarrow X$ is a surjective local homeomorphism, we have $T^{n}(A) = A = T^{-n}(A)$. 
    Hence, $T^n|_A \colon A\to A$ is a surjective local homeomorphism, which is minimal as $A$ is the maximal element in $\mathcal{F}$. 
	It remains to prove that $A$ is open. 
	Consider the closed set 
    \[C\coloneqq \bigcup_{k=0}^{n-1}T^{-k}(A).\] 
    We claim that $C$ is $T$-invariant. Using $T^{-n}(A)=A$ at the last equality, we have 
    \[T^{-1}(C)=T^{-1}\left(\bigcup_{k=0}^{n-1}T^{-k}(A)\right)=\bigcup_{k=1}^nT^{-k}(A)=C. \]
    Moreover, using $T^n(A)=A$ and thus $T(A)\subset T^{-(n-1)}(A)$ at the last step, we obtain
    \[T(C)=T\left(\bigcup_{k=0}^{n-1}T^{-k}(A)\right)\subset \bigcup_{k=-1}^{n-2}T^{-k}(A)\subset \bigcup_{k=0}^{n-1}T^{-k}(A)=C. \]
    By minimality of $T$, it follows that $X=C$.
    By (the finite version of) the Baire category theorem, there exists $0\leq k\leq n-1$ such that $T^{-k}(A)$ has non-empty interior.
    Since $T^k$ is a local homeomorphism, it follows that $A$ has non-empty interior which we denote by $\mathrm{int}(A)$.
    Using $T^n(A)=A=T^{-n}(A)$ and that $T^n$ is a local homeomorphism, we conclude that 
    \[T^n(\mathrm{int}(A))=\mathrm{int}(A)=T^{-n}(\mathrm{int}(A))\not=\emptyset.\]
    By minimality of $T^n|_A$, we conclude that $A=\mathrm{int}(A)$.
\end{proof} 

The next lemma allows us to transfer dynamical comparison with no invariant measure from a minimal open subgroupoid to the ambient groupoid. 

\begin{lemma}\label{lem-open-subgroupoid} 
    Let $G$ be a minimal Hausdorff \'etale groupoid with compact unit space and no isolated points.
    Suppose that there is a non-empty minimal open subgroupoid $H\subset G$ with compact unit space satisfying dynamical comparison with no invariant measures.
    Then $G$ satisfies dynamical comparison with no invariant measures.
\end{lemma}
\begin{proof}
    Since $G$ is minimal, there exist open bisections $B_0,\dotsc,B_m\subset G$ satisfying $r(B_i)\subset H^{(0)}$ for all $i=0,\dotsc,m$ and
    $G^{(0)} = \bigcup_{i=0}^ms(B_i)$.
    Let $U\subset G^{(0)}$ be a non-empty open subset.
    Using minimality of $G$ again, there is an open bisection $C\subset G$ such that $\emptyset \not= s(C)\subset H^{(0)}$ and $r(C)\subset U$. 
    Since $G^{(0)}$ has no isolated points, there are pairwise disjoint non-empty open subsets $U_0,\dotsc,U_m\subset s(C)$.
    As $H$ satisfies dynamical comparison with no invariant measure, we have $H^{(0)}\prec U_{i}$ for $i = 0, \dots, m$. 
    Hence, for each $i = 0, \dots, m$, there exist open bisections $D_{i,0},\dotsc,D_{i,n_i}\subset H$ such that $H^{(0)}=\bigcup_{j=0}^{n_i}s(D_{i,j})$, and $r(D_{i,0}),\dotsc,r(D_{i,n_i})\subset U_i$ are pairwise disjoint. 
    It follows that $\{CD_{i,j}B_i\mid i=0,\dotsc,m, j=0,\dotsc,n_i\}$ is a set of open bisections in $G$ satisfying
    \[G^{(0)}\subset \bigcup_{i=0}^ms(B_i)=\bigcup_{i=0}^m\bigcup_{j=0}^{n_i}s(D_{i,j}B_i)=\bigcup_{i=0}^m\bigcup_{j=0}^{n_i}s(CD_{i,j}B_i), \] 
    and the sets $r(CD_{i,j}B_i)\subset U$ for $i=0,\dotsc,m$ and $j=0,\dotsc,n_i$ are pairwise disjoint.
    This proves $G^{(0)}\prec U$. 
    In particular, no $G$-invariant probability measure $\mu$ could exist on $G^{(0)}$, because $G^{(0)}\prec U$ would force $\mu (U) = 1$ for every nonempty open set $U$, contradicting the existence of two disjoint non-empty open subsets in $G^{(0)}$. 
    Thus $G$ has dynamical comparison with no invariant measures. 
\end{proof} 

By combining the above lemmas, we obtain the following reduction of the main theorem.

\begin{cor}\label{cor-reduction}
    The following are equivalent.
    \begin{enumerate}
        \item\label{item1reduction} For every compact metrizable space $X$ of finite Lebesgue covering dimension and every minimal surjective local homeomorphism $T\colon X\to X$, the Deaconu--Renault groupoid $G_T$ satisfies dynamical comparison;
        \item\label{item2reduction} For every compact metrizable space $X$ of finite Lebesgue covering dimension and every minimal surjective local homeomorphism $T\colon X\to X$ satisfying $|T^{-1}(x)|\geq 2$ for all $x\in X$, the Deaconu--Renault groupoid $G_T$ satisfies dynamical comparison with no invariant measures. 
    \end{enumerate}
\end{cor}
\begin{proof}
    Suppose first that \ref{item2reduction} holds. 
    Let $X$ be a compact metrizable space of finite Lebesgue covering dimension.
    Let $T\colon X\to X$ be a minimal surjective local homeomorphism. 
    If $T$ is injective, then dynamical comparison follows from \cite{Downarowicz2023,Naryshkin22}.
    Suppose that $T$ is not injective. By \autoref{lem:unif_non_inj}, there is an integer $n\geq 1$ such that $|T^{-n}(x)|\geq 2$ for all $x\in X$. 
    By \autoref{lem:min_n_set}, there is a compact open subset $A\subset X$ such that $T^n|_A\colon A\to A$ is a surjective minimal local homeomorphism. 
    Moreover, $\dim(A)\leq \dim(X)<\infty$ since $A\subset X$ is closed. 
    Thus, $G_{T^n|_A}\subset G_T$ is an open, minimal subgroupoid satisfying dynamical comparison with no invariant measures by assumption \ref{item2reduction}. 
    Moreover, $X$ contains no isolated points: the set of isolated points is open and invariant, and minimality of $T$ would otherwise force $X$ to be a finite discrete space, contradicting surjectivity and non-injectivity of $T$. 
    It follows from \autoref{lem-open-subgroupoid} that $G_T$ satisfies dynamical comparison with no invariant measures, verifying \ref{item1reduction}.

    Now suppose that \ref{item1reduction} holds and that $T\colon X\to X$ is a minimal surjective local homeomorphism such that $|T^{-1}(x)|\geq 2$ for all $x\in X$.
    It remains to prove that $G_T$ has no invariant Borel probability measures. 
    Suppose, for contradiction, that there is a $G_T$-invariant Borel probability measure $\mu$ on $X$.
    By assumption, for every $y\in X$, there exist distinct $x^1,x^2\in X$ with $T(x^1)=y=T(x^2)$. 
    Since $T$ is a local homeomorphism, we can find an open neighbourhood $V_{y}$ of $y$ and disjoint open neighbourhoods $U^1_y, U^2_y$ of $x^{1}, x^{2}$ respectively, such that $T|_{U^i_y}\colon U^i_y\xrightarrow{\cong} V_y$ is a homeomorphism for $i= 1, 2$. 
    By compactness, we can find $y_0,\dotsc,y_n\in X$ such that $X = V_{y_0}\cup \ldots \cup V_{y_n}$. 
    Define pairwise disjoint Borel sets by $V_{0} = V_{y_{0}}$ and 
    \[V_j\coloneqq V_{y_j}\setminus (V_{y_0}\cup \ldots \cup V_{y_{j-1}}), \quad j = 1, \dots, n, \]
    and pairwise disjoint Borel sets 
    \[U^i_j\coloneqq T|{U^i_{y_j}}^{-1}(V_j), \quad i = 1,2 \text{ and } j = 0, \dots, n.\]
    Moreover, invariance on open bisections implies invariance on their Borel restrictions, and hence $\mu (U_{j}^{i}) = \mu (V_{j})$. Therefore, we have  
    \[\mu(X)\geq \sum_{j=0}^n(\mu(U^1_j)+\mu(U^2_j))=2\sum_{j=0}^{n}\mu(V_j)=2\mu(X),\]
    contradicting $\mu(X)=1$.
\end{proof}

\section{Embedding partial actions of free groups}\label{sec-partial}
In this section, we prove the main theorem for local homeomorphisms with ``large fibers'' (see \autoref{cor-reduction}), using 
the thin boundary property from \autoref{thm-Lebesgue-thin}. 
We also give a more direct proof in the zero-dimensional case without using the thin boundary property.
Our main technical input is a variation of the ``paradoxical towers'' technique from \cite{Gardella2023} for partial actions of non-abelian free groups with ``large domains''. 
The construction of the specific paradoxical towers needed here explicitly uses the north-south dynamics of the action of the free group on its Gromov boundary. 
The same construction applies to torsion-free Gromov-hyperbolic groups, and we formulate it in that generality.  
For our application to Deaconu--Renault groupoids, however, only non-abelian free groups will play a role. 

Recall that a finitely generated group $\Gamma=\langle S\rangle$ is called \emph{hyperbolic} if its Cayley graph $\mathrm{Cay}(\Gamma,S)$ is Gromov-hyperbolic as a metric space. 
To each hyperbolic group $\Gamma$, one associates a \emph{Gromov boundary} $\partial \Gamma$ which is a compact metrizable $\Gamma$-space defined by equivalence classes of quasi-geodesic rays in $\mathrm{Cay}(\Gamma,S)$. 
Both hyperbolicity and the Gromov boundary are independent of the choice of $S$. 
An element $g\in \Gamma$ is called \emph{loxodromic} if it has infinite order and its action on $\partial \Gamma$ has exactly two fixed points. 
The fixed points of a loxodromic element $g$ consist of an \emph{attracting fixed point} $g^+\in \partial \Gamma$ and a \emph{repelling fixed point} $g^-\in \partial \Gamma$. 
For every open neighbourhood $U^-$ of $g^-$ and every open neighbourhood $U^+$ of $g^{+}$, we have $g^n(\partial \Gamma\setminus U^-)\subset U^+$ for sufficiently large positive integers $n\geq 1$. 
Two loxodromic elements are called \emph{independent} if their sets of fixed points are disjoint.
Every hyperbolic group $\Gamma$ has a maximal finite normal subgroup which is called the \emph{finite radical}. 
A hyperbolic group $\Gamma$ is called \emph{non-elementary} if it is not virtually cyclic (equivalently if it is non-amenable).
We summarize the properties of the Gromov boundary relevant for this paper below and refer the reader to \cite{Bridson1999} for a detailed overview. 

\begin{lemma}\label{lem-hyperbolic-properties}
    Let $\Gamma$ be a non-elementary hyperbolic group with trivial finite radical. 
    Then 
    \begin{enumerate}
        \item\label{item-hyperbolic1} The action $\Gamma\curvearrowright \partial \Gamma$ is topologically free, i.e. for every $g\in \Gamma \setminus \{1\}$, the set $\{x\in \partial \Gamma\mid gx=x\}\subset \partial \Gamma$ has empty interior;
        \item\label{item-hyperbolic2} If $\Gamma=\langle a,b\rangle$ is a free group on two generators and $\langle a,b\rangle^+$ denotes the semigroup generated by $a$ and $b$, then the sets $\{w^+ \mid w\in \langle a,b\rangle^+\}\subset \partial \Gamma$ and  $\{w^- \mid w\in \langle a,b\rangle^+\}\subset \partial \Gamma$ are infinite.
    \end{enumerate}
\end{lemma}
\begin{proof}
    See \cite[Proposition~4.1]{Abbott2019a} and \cite[8.2.E]{Gromov}.
\end{proof}

The following lemma is inspired by the construction of \emph{strong paradoxical towers} in $F_2$ as explained in \cite[Proposition~3.2]{Gardella2023}. 
The lemma does not generalize to arbitrary \emph{groups with paradoxical towers} as in \cite{Gardella2023} since we need to be able to choose the elements $h_i$ from a prescribed subsemigroup of $\Gamma$. 
\begin{lemma}\label{lem:paradoxical_towers}
    Let $\Gamma$ be a non-elementary torsion-free hyperbolic group, and let $a,b\in \Gamma$ be independent loxodromic elements. 
    Let $F\subset \Gamma$ be a finite subset. 
    Then there exist subsets $A_1,A_2,A_3\subset \Gamma$ and elements $h_1,h_2,h_3\in\langle a,b\rangle^+\subset \Gamma$ 
    such that
	\begin{enumerate}
		\item the sets $gA_i$ for $g\in F$ and $i=1,2,3$ are pairwise disjoint, and
		\item the sets $\Gamma\setminus h_iA_i$ for $i=1,2,3$ are pairwise disjoint.
	\end{enumerate}
\end{lemma}
\begin{proof}
    By Gromov's ping-pong argument (see \cite[8.2.F]{Gromov} for instance), after replacing $a$ and $b$ by sufficiently large positive powers, we may assume that the subgroup $\langle a,b\rangle\subset\Gamma$ is free and that the Gromov boundary of the free subgroup $\langle a,b\rangle\cong F_2$ embeds into $\partial \Gamma$. 
    It therefore follows from \autoref{lem-hyperbolic-properties}~\ref{item-hyperbolic2} that the sets
    $\{w^+ \mid w\in \langle a,b\rangle^+\}\subset \partial \Gamma$ and  $\{w^- \mid w\in \langle a,b\rangle^+\}\subset \partial \Gamma$ are infinite.
   
    We claim that there exist $x_1,x_2,x_3\in \langle a,b\rangle^+$ such that the points
    \[
        g x_i^- \in \partial \Gamma,
        \text{ for }  g\in F,\ i=1,2,3,
    \]
    are pairwise distinct, and such that $x_1^+,x_2^+,x_3^+$ are pairwise
    distinct.  
    Indeed, since $\Gamma$ is torsion-free, every nontrivial element of $\Gamma$ is loxodromic and thus fixes exactly two points of $\partial\Gamma$. Hence the set of points fixed by at least one $g\in F^{-1}F\setminus\{1\}$ is finite. Moreover, the words $w_n=ab^n$, for $n\geq 1$, have pairwise distinct attracting fixed points and pairwise distinct repelling fixed points in the embedded boundary of $\langle a,b\rangle$. We can therefore inductively choose $x_i$ from among the $(w_n)_{n\geq 1}$ such that $x_i^-$ is not fixed by any $g\in F^{-1}F\setminus\{1\}$ and such that $x_i^-\notin \{gx_{j}^-\mid g\in F^{-1}F, 1\leq j\leq i-1\}$ and $x_i^+\notin \{x_1^+,\dotsc,x_{i-1}^+\}$. 
    At each step, these conditions exclude only finitely many values of $n$. This proves the claim.
    
    Choose open neighbourhoods $U_i\subset \partial \Gamma$ of $x_i^-$, for
    $i=1,2,3$, such that the sets $\{gU_i\mid  g\in F, \ i=1,2,3\}$ are pairwise disjoint. 
    Also choose pairwise disjoint open neighbourhoods $V_i\subset \partial \Gamma$ of $x_i^+$, for $i=1,2,3$.
    Since each $x_i$ is loxodromic, there is
    a positive integer $N\geq 1$ such that
\[
        x_i^N(\partial \Gamma\setminus U_i)\subset V_i, \text{ for } i=1,2,3.
\]
    Set $h_i=x_i^N$, for $i=1,2,3$.
    Since the $V_i$'s are pairwise disjoint, the sets
    \[\{\partial \Gamma\setminus h_iU_i\mid i=1,2,3\}\]
    are also pairwise disjoint.
    Fix $z\in \partial \Gamma$, and define
\[
        A_i\coloneqq \{g\in \Gamma \mid gz\in U_i\}, \text{ for }  i=1,2,3.
\]
    The sets $A_1,A_2,A_3$ and the elements $h_1,h_2,h_3$ satisfy the desired conditions.
\end{proof} 

To prove dynamical comparison for Deaconu--Renault groupoids with ``large fibers'' (the statement in Corollary~\ref{cor-reduction}\ref{item2reduction}), we use Steinberg's theorem \cite[Theorem~2.2]{Steinberg2026}, based on an observation of Gilles de Castro. 
It realizes Deaconu--Renault groupoids with zero-dimensional compact unit space as transformation groupoids of partial actions of free groups. 
We first recall the definition of partial group actions, introduced by Exel \cite{Exel1994}.   

\begin{defi}[Partial actions]
    Let $\Gamma$ be a discrete group and let $X$ be a locally compact Hausdorff space.
    A \emph{partial action} $ \Gamma \overset{\alpha}\curvearrowright X$ consists of a collection of open subsets $\{D_g\subset X\}_{g\in \Gamma}$ and homeomorphisms $\{\alpha_g\colon D_{g^{-1}}\to D_g\}_{g\in \Gamma}$ such that 
    \begin{enumerate}
        \item $D_1=X$ and $\alpha_1=\id$;
        \item For every $g,h\in \Gamma$ we have an inclusion $\alpha_g\circ\alpha_h\subset \alpha_{gh}$ of partially defined maps. 
        More precisely, $D_{h^{-1}}\cap \alpha_h^{-1}(D_{g^{-1}})\subset D_{(gh)^{-1}}$ and $\alpha_g\circ \alpha_h(x)=\alpha_{gh}(x)$ for all $x\in D_{h^{-1}}\cap \alpha_h^{-1}(D_{g^{-1}})$.
    \end{enumerate}
    The \emph{transformation groupoid} associated $\alpha$ is given by 
    \[\Gamma\ltimes X\coloneqq \bigcup_{g\in \Gamma}\{g\}\times D_{g^{-1}}\subset \Gamma\times X\]
    with unit space $X\cong \{1\}\times X$ and range, source, and multiplication maps given by 
    \[r(g,x)=gx, \quad s(g,x)=x,\quad (h,gx)(g,x)=(hg,x).\]
\end{defi}

A partial action $\Gamma \overset{\alpha}\curvearrowright X$ is called \emph{amenable} if the transformation groupoid $\Gamma\ltimes X$ is an amenable groupoid in the sense of \cite{AnantharamanDelaroche2000}. 
Equivalently, there is a net of continuous maps $\mu_i\colon X\to \Prob(\Gamma)$, where $\Prob(\Gamma)$ denotes the set of probability measures on $\Gamma$ with the $\ell^1$-topology, such that for every finite set $F\subset \Gamma$ and every family of compact subsets $\{K_g\subset D_g\}_{g\in F}$, we have 
\[\lim_{i}\sup_{x\in K_g}\|\mu_i(g^{-1}x)-g^{-1}(\mu_i(x))\|_1= 0,\quad g\in F.\]

Assume further that $\Gamma$ is a finitely generated group with word length $\ell$. A word $g=g_1 \dotsb g_n\in \Gamma$ is called \emph{reduced} if we have $\ell(g)=\ell(g_1)+\dotsb+\ell(g_n)$. 
A partial action $\Gamma\overset{\alpha}\curvearrowright X$ is called \emph{semi-saturated} if for every reduced word $g=g_1 \dotsb  g_n\in \Gamma$, we have $\alpha_{g_1}\circ \dotsb \circ \alpha_{g_n}=\alpha_g$. 
Denoting by $F_d=\langle a_1,\dotsc,a_d\rangle$ the free group on $d$ generators, any choice of homeomorphisms $T_i\colon U_i\xrightarrow{\cong}V_i, i=1,\dotsc,d$ for open subsets $U_1,V_1,\dotsc,U_d,V_d\subset X$ uniquely defines a semi-saturated partial action $F_d\overset{\alpha}\curvearrowright X$ by requiring $\alpha_{a_i}=T_i$ for all $i=1,\dotsc,d$. 
We 
say that a partial action $F_d\overset{\alpha}\curvearrowright X$ is \emph{orthogonal} if the domains $D_{a_i^{-1}}$ for $i=1,\dotsc,d$ are pairwise disjoint. 

Before proving the main theorem of the section, we provide a short, self-contained proof in the zero-dimensional case, without invoking Theorem~\ref{thm-Lebesgue-thin}. 
Although covered by the general argument, this special case illustrates the main ideas and motivates the higher-dimensional case. 
We combine Lemma~\ref{lem:paradoxical_towers} on paradoxical towers with techniques developed in \cite{Gardella2023}, to prove comparison for partial actions of free groups. 
In this setting, the requirement that two independent loxodromic elements have large domains in Lemma~\ref{lem:paradoxical_towers} corresponds precisely to the ``large fibers'' assumption for the associated Deaconu--Renault groupoid. 

\begin{theorem} 
    Let $X$ be a totally disconnected compact Hausdorff space\footnote{Although \autoref{cor-reduction} assumes metrizability, this assumption is not needed for the special case. Metrizability is only used in \autoref{sec-dimension} for the higher-dimensional case.}, and let $T\colon X\to X$ be a minimal surjective local homeomorphism. Then $G_T$ satisfies dynamical comparison.
\end{theorem} 

\begin{proof}
    By \autoref{cor-reduction}, we may assume that $|T^{-1}(x)|\geq 2$ for all $x\in X$. 
    By the assumptions on $T$, for every $x\in X$, we may find an open neighbourhood $V_x$ and pairwise disjoint open sets $U_x^i$ for $i=1,2$ such that $T$ restricts to homeomorphisms $T|_{U_x^i} \colon U_x^i\xrightarrow{\cong}V_x$. 
    By compactness of $X$, there exist finitely many $x_1,\dotsc,x_n\in X$ such that $X\subset V_{x_1}\cup \dotsc \cup V_{x_n}$. 
    Since $X$ is totally disconnected, we moreover assume that the sets $V_{x_1},\dotsc,V_{x_n}$ are clopen and pairwise disjoint. 
    We define clopen sets 
    \[A_i\coloneqq \bigsqcup_{j=1}^nU_{x_j}^i,\text{ for }i=1,2.\]
    Since $T$ is a local homeomorphism and $X$ is a compact totally disconnected space, we can extend $A_1\sqcup A_2$ to a clopen partition $X=A_1\sqcup A_2\sqcup \ldots\sqcup A_d$ for some $d\geq 2$ such that $T|_{A_i}$ is injective for each $i=1,\ldots, d$. 
    Denote by $F_d=\langle a_1,\dotsc,a_d\rangle$ the free group on $d$ generators. 
    We define a semi-saturated orthogonal partial action $F_d\overset{\alpha}\curvearrowright X$ by declaring $\alpha_{a_i}=T|_{A_i}$. 
    It follows from the proof of \cite[Theorem~2.2]{Steinberg2026} that $G_T\cong F_d\ltimes X$. 
    In particular, $\alpha$ is minimal and amenable (see \cite[Corollary~9.7]{ExelSteinberg}) and it suffices to prove dynamical comparison for $\alpha$.
    
    As illustrated in Corollary~\ref{cor-reduction}, the condition $|T^{-1}(x)|\geq 2$ for all $x\in X$ implies that $G_{T}$ has no invariant measure. 
    Thus, it suffices to prove that $\alpha$ has dynamical comparison with no invariant measure, that is, $X\prec U$ for any non-empty open subset $U$ of $X$. 
    Since $X$ is zero-dimensional we may without loss of generality assume that $U$ is clopen. 
    By minimality, there is a finite set $F\subset F_d$ such that 
    \begin{equation}\label{eq-cover-by-F0}
         X \subset \bigcup_{g\in F}\alpha_{g^{-1}}(D_g\cap U).
    \end{equation}
    Without loss of generality, we may assume that $1\in F$.
    By \autoref{lem:paradoxical_towers} there are subsets $A_1,A_2,A_3\subset F_{d}$ and elements $h_1,h_2,h_3\in \langle a_1,a_2\rangle^+$ such that
	\begin{enumerate}
		\item the sets $gA_i$ for $g\in F$ and $i=1,2,3$ are pairwise disjoint and
		\item the sets $F_{d}\setminus h_iA_i$ for $i=1,2,3$ are pairwise disjoint.
	\end{enumerate}
    Moreover, denoting by $\{D_g\subset X\}_{g\in F_d}$ the domains of $\alpha$, we have $D_{h_i}=D_{a_1}=D_{a_2}=X$ for $i=1,2,3$.
    Let $\varepsilon\coloneqq \tfrac{1}{13}$. 
    By amenability of $\alpha$, there is a continuous map $\mu\colon X\to \Prob(F_d)$ such that 
    \[\sup_{x\in {D_{g^{-1}}}}\|\mu(\alpha_gx)-g\mu(x)\|_1<\varepsilon,\text{ for all }g\in F,\]
    and 
    \[\sup_{x\in X}\|\mu(\alpha_{h_i^{-1}}(x))-h_i^{-1}\mu(x)\|_1<\varepsilon, \text{ for all }i=1,2,3.\]
    Define 
    \[V_i\coloneqq \{x\in X\mid \mu(x)(A_i)>\tfrac{1}{2}+\varepsilon\}, \text{ for all }i=1,2,3, \]
    and
    \[W_i\coloneqq \{x\in X\mid \mu(x)(F_{d}\setminus h_iA_i)<\tfrac{1}{2}-2\varepsilon\}, \text{ for all } i=1,2,3.\]
    
    Since the complements $F_{d}\setminus h_iA_i$ for $i=1,2,3$ are pairwise disjoint, for every probability measure on $F_{d}$, at least one of these complements has measure at most $\frac{1}{3}$. 
    Since $\frac 1 3<\frac 1 2-2\varepsilon$, it follows that 
    \begin{equation}\label{eq-Wcover}
        X=W_1\cup W_2\cup W_3.
    \end{equation}
    By our choice of $\mu$ and $W_i$, we have $\mu(\alpha_{h_i^{-1}}(x))(A_i)>\frac 1 2 +\varepsilon$ for all $i=1,2,3$ and $x\in  W_i$ so that 
    \begin{equation}\label{eq-X<V0}
    \alpha_{h_i^{-1}}( W_i)\subset V_i,\text{ for all }i=1,2,3.
    \end{equation}
       
    Analogously, since the sets $gA_i\subset F_d$ for $g\in F$ and $i=1,2,3$ are pairwise disjoint, at most one of them can have probability measure $>\frac 1 2$.
    Since we chose $\mu$ and $V_i$ such that $\mu(\alpha_g(x))(gA_i)>\frac 1 2$ for all $g\in F, x\in V_i\cap \alpha_{g^{-1}}(D_g\cap U)$, and $i=1,2,3$, the sets 
    \begin{equation}\label{eq-V<U0}
        \alpha_g( V_i\cap  \alpha_{g^{-1}}(D_g\cap U))\subset U\text{ for }g\in F\text{ and }i=1,2,3
    \end{equation}
    are pairwise disjoint. 

    Now the combination of \eqref{eq-Wcover} and \eqref{eq-X<V0} proves $X\prec V_1\cup V_2\cup V_3$, while the combination of \eqref{eq-cover-by-F0} and \eqref{eq-V<U0} proves $V_1\cup V_2\cup V_3\prec U$.
    Thus, $X\prec U$ as desired.
\end{proof}

For Deaconu--Renault groupoids with higher-dimensional unit spaces, the de Castro--Steinberg realization is not directly applicable. 
We therefore prove the following technical lemma, extending the comparison result of \cite{Gardella2023} to minimal {\'e}tale groupoids containing suitably ``large'' open subgroupoids arising from amenable partial actions of non-amenable hyperbolic groups. 
The \emph{thin boundary property} controls the boundary of the embedded partial action; see \autoref{def-thin} for the definition of $G$-thin subsets. 

\begin{theorem}\label{thm-large-hyperbolic-subgroupoid}
    Let $G$ be a minimal Hausdorff \'etale groupoid with compact unit space $G^{(0)}$ and no isolated points.
    Assume that there is an amenable partial action
    $\Gamma\overset{\alpha}\curvearrowright X$ of a non-elementary torsion-free hyperbolic group $\Gamma$ satisfying the following conditions: 
    \begin{enumerate}
        \item $G$ contains $\Gamma\ltimes X$ as an open subgroupoid with $G^{(0)}=X$;
        \item\label{item-thin-basis} The topology of $G$ has a basis of relatively compact open bisections $B\subset G$ such that $r\left(\overline B\setminus (\Gamma\ltimes X)\right)$ is $G$-thin and such that $B$ only intersects finitely many bisections of the form $\{g\}\times D_{g^{-1}}$ for $g\in \Gamma$;
        \item There are independent loxodromic elements $a,b\in \Gamma$ such that $X\setminus D_a$ and $X\setminus D_b$ are $G$-thin subsets of $X$. 
    \end{enumerate}
    Then $G$ satisfies dynamical comparison. 
\end{theorem}

We first record a lemma needed; see \autoref{notation-theta} for the definition of $\theta_B$. 
\begin{lemma}\label{lem-thin}
    Let $G$ be a Hausdorff \'etale groupoid such that $G^{(0)}$ is compact and has no isolated points. 
Then the following hold: 
\begin{enumerate}
    \item\label{item-thin-1} Finite unions and closed subsets of $G$-thin sets are $G$-thin.
    \item\label{item-thin-3} Let $E\subset G^{(0)}$ be a $G$-thin subset and let $B_1,\dotsc,B_n$ be open bisections. Then any compact subset $K\subset \bigcup_{i=1}^n\theta_{B_i}( E)$ is $G$-thin.
\end{enumerate}
Now assume further that $G=G_T$ for a minimal local homeomorphism $T\colon X\to X$. 
\begin{enumerate}\setcounter{enumi}{2}
    \item\label{item-thin-2} If $E\subset X$ is a $G_T$-thin set, then $T(E)$ and $T^{-1}(E)$ are $G_T$-thin as well. 
\end{enumerate}
\end{lemma}
\begin{proof}
    To prove \ref{item-thin-1}, let $E_1,\dotsc,E_n\subset G^{(0)}$ be $G$-thin sets and let $U\subset G^{(0)}$ be a non-empty open set.
    Since $G^{(0)}$ has no isolated points, there are pairwise disjoint non-empty open sets $U_1,\dotsc,U_n\subset U$. 
    By assumption, we have $E_i\prec U_i$ for all $i=1,\dotsc,n$ and thus $E_1\cup \ldots \cup E_n\prec U$. 
    It is immediate that closed subsets of $G$-thin sets are $G$-thin. 

    To prove \ref{item-thin-3}, let $K\subset \bigcup_{i=1}^n\theta_{B_i}(E)$ be compact and let $U\subset G^{(0)}$ be open and non-empty. 
    Since $G^{(0)}$ has no isolated point, we may choose pairwise disjoint non-empty open subsets $U_{1}, \cdots, U_{n} \subset U$. 
    For each $1\leq i\leq n$, let $C_{i,1}, \dots, C_{i,k_{i}}$ be open bisections witnessing $E\prec U_{i}$. 
    Then the bisections $C_{i, j} B_{i}^{-1}$ for $1\leq i\leq n$ and $1\leq j \leq k_{i}$ cover $K$ by their sources, and their ranges are pairwise disjoint subsets of $U$. 
    Hence $K\prec U$ and $K$ is $G$-thin. 

    To prove \ref{item-thin-2}, let $E\subset X$ be a $G_T$-thin set. By compactness, there is an open cover $X\subset U_1\cup\dotsb\cup U_n$ such that $T|_{U_i}\colon U_i\to T(U_i)$ is a homeomorphism for each $i=1,\dotsc,n$.
    By normality and compactness, we may find open subsets $V_i\subset \overline{V_i}\subset U_i$ such that $X\subset V_1\cup \dotsb \cup V_n$.
    Consider the open bisections
    \[B_i\coloneqq \{(T(x),-1,x)\mid x\in U_i\}\subset G_T,\text{ for all }i=1,\dotsc,n.\]
    Then the sets 
    \begin{align*}
    T(E)&\subset \theta_{B_1}\left(\overline V_1\cap E\right)\cup\dotsb\cup \theta_{B_n}\left(\overline V_n\cap E\right),\\
    T^{-1}(E)&\subset \theta_{B_1^{-1}}\left(T\left(\overline V_1\right)\cap E\right)\cup\dotsb\cup \theta_{B_n^{-1}}\left(T\left(\overline V_n\right)\cap E\right)
    \end{align*}
    are $G_T$-thin by \ref{item-thin-3}. 
\end{proof}

\begin{proof}[Proof of \autoref{thm-large-hyperbolic-subgroupoid}]
    Let $U\subset X$ be a non-empty open subset. 
    Since $X$ has no isolated points, there are non-empty disjoint open subsets $U_1,U_2\subset U$.
    Since $G$ is minimal and $X$ is compact, we can find relatively compact basic open bisections $B_1,\dotsc,B_n\subset G$ as in condition \ref{item-thin-basis} of the theorem such that 
    \[X=\bigcup_{i=1}^n \theta_{B_i}(U_1).\]
    Using \autoref{lem-thin}~\ref{item-thin-1}, we deduce that the closed set 
    \[E\coloneqq X\setminus \bigcup_{i=1}^n\theta_{B_i\cap (\Gamma\ltimes X)}(U_1)\subset \bigcup_{i=1}^n\theta_{B_i}(U_1)\setminus\theta_{B_i\cap (\Gamma\ltimes X)}(U_1)\subset \bigcup_{i=1}^nr\big(\overline B_i\setminus (\Gamma\ltimes X)\big)\]
    is $G$-thin.
    By the finite-intersection assumption in \ref{item-thin-basis} of the theorem, there is a finite set $F\subset \Gamma$ such that 
    \begin{equation}\label{eq-cover-by-F}
         X\setminus E= \bigcup_{i=1}^n\theta_{B_i\cap (\Gamma\ltimes X)}(U_1) \subset \bigcup_{g\in F}\alpha_{g^{-1}}(D_g\cap U_1).
    \end{equation}
    Without loss of generality, we may assume that $1\in F$. 
    By \autoref{lem:paradoxical_towers} there are subsets $A_1,A_2,A_3\subset \Gamma$ and elements $h_1,h_2,h_3\in \langle a,b\rangle^+$ such that
	\begin{enumerate}
		\item the sets $gA_i$ for $g\in F$ and $i=1,2,3$ are pairwise disjoint and
		\item the sets $\Gamma \setminus h_iA_i$ for $i=1,2,3$ are pairwise disjoint.
	\end{enumerate}
    We claim that $X\setminus D_{h_i}$ is $G$-thin for each $i=1,2,3$. 
    Indeed, fix $i=1,2,3$ and write $h_i=g_{i,1}\dotsb g_{i,n_i}$ with $n_i\geq 0$ and $g_{i,j}\in \{a,b\}$ for $1\leq j\leq n_i$. 
    Define $g_{i,0}\coloneqq 1$. 
    If $x\in X\setminus D_{h_i}$, then there must be a $j=0,\dotsc,n_i-1$ such that $(\alpha_{g_{i,0}}\circ \dotsb \circ \alpha_{g_{i,j}})^{-1}$ is defined at $x$, but $(\alpha_{g_{i,0}}\circ \dotsb \circ \alpha_{g_{i,j}})^{-1}(x)\notin D_{g_{i,j+1}}$. 
    Noting that $X\setminus D_{g_{i,j}}$ is contained in the $G$-thin set $E'\coloneqq (X\setminus D_a)\cup (X\setminus D_b)$ for all $j=1,\dotsc,n_i$, we conclude that the set
    \[X\setminus D_{h_i}\subset \bigcup_{j=0}^{n_i} \alpha_{g_{i,1}\dotsb g_{i,j}}(E'\cap D_{(g_{i,1}\dotsb g_{i,j})^{-1}})\] 
    is $G$-thin by \autoref{lem-thin}.

    Note that the set 
    \[ \alpha_{h_i}\big(E\cap D_{h_i^{-1}}\big)\cup (X\setminus D_{h_i})\subset X\]
    is closed for each $i=1,2,3$ since its complement is given by $\alpha_{h_i}(D_{h_i^{-1}}\setminus E)$. 
    It is moreover $G$-thin. 
    Indeed, if $O\subset X$ is non-empty and open and $O_1,O_2\subset O$ are disjoint non-empty open subsets, compose the bisections witnessing $E\prec O_{1}$ with the bisection for $\alpha_{h_{i}}^{-1}$, and combine them with the bisections witnessing $X\setminus D_{h_i}\prec O_2$. 
    The resulting family witnesses $\alpha_{h_i}\big(E\cap D_{h_i^{-1}}\big)\cup (X\setminus D_{h_i})\prec O$. 
    It thus follows from \autoref{lem-thin}~\ref{item-thin-1} that
    \[E''\coloneqq E\cup \bigcup_{i=1}^{3} (\alpha_{h_i}\big(E\cap D_{h_i^{-1}}\big)\cup (X\setminus D_{h_i}))\]
    is $G$-thin so that $E''\prec U_2$. 
    By normality and compactness, there is an open neighbourhood $N$ of $E''$ such that 
    $\overline N\prec U_2$. 

    It remains to show that $X\setminus N\prec U_1$. 
    By construction, we have $\alpha_{h_i^{-1}}(X\setminus E'')\subset X\setminus E$ for all $i=1,2,3$ so that \eqref{eq-cover-by-F} implies 
    \[
        X\setminus N\cup \bigcup_{i=1}^{3}\alpha_{h_i^{-1}}(X\setminus N) \subset X\setminus E\subset \bigcup_{g\in F}\alpha_{g^{-1}}(D_g\cap U_1).
    \]
    By normality and compactness, we can find open subsets $O_g\subset \overline{O_g}\subset \alpha_{g^{-1}}(D_g\cap U_1)$ such that 
    \begin{equation}\label{eq-XN-in_O_g}
        X\setminus N\cup \bigcup_{i=1}^{3}\alpha_{h_i^{-1}}(X\setminus N)\subset \bigcup_{g\in F}O_g.
    \end{equation}
    Let $\varepsilon\coloneqq \tfrac{1}{13}$. 
    By amenability of $\alpha$, there is a continuous map $\mu\colon X\to \Prob(\Gamma)$ such that 
    \[\sup_{x\in \overline{O_g}}\|\mu(\alpha_gx)-g\mu(x)\|_1<\varepsilon,\text{ for all }g\in F,\]
    and 
    \[\sup_{x\in X\setminus N}\|\mu(\alpha_{h_i^{-1}}(x))-h_i^{-1}\mu(x)\|_1<\varepsilon, \text{ for all }i=1,2,3.\]
    Define 
    \[V_i\coloneqq \{x\in X\mid \mu(x)(A_i)>\tfrac{1}{2}+\varepsilon\}, \text{ for all }i=1,2,3, \]
    and
    \[W_i\coloneqq \{x\in X\mid \mu(x)(\Gamma \setminus h_iA_i)<\tfrac{1}{2}-2\varepsilon\}, \text{ for all } i=1,2,3.\]
    
    Since the complements $\Gamma \setminus h_{i} A_{i}$ for $i=1,2,3$ are pairwise disjoint, for every probability measure on $\Gamma$, at least one of these complements has measure at most $\frac{1}{3}$.   
    Since $\frac 1 3<\frac 1 2-2\varepsilon$, it follows that 
    \[X=W_1\cup W_2\cup W_3.\]
    By our choice of $\mu$ and $W_i$, we have $\mu(\alpha_{h_i^{-1}}(x))(A_i)>\frac 1 2 +\varepsilon$ for all $i=1,2,3$ and $x\in (X\setminus N) \cap W_i$ so that $\alpha_{h_i^{-1}}((X\setminus N)\cap W_i)\subset V_i$.
    Combined with \eqref{eq-XN-in_O_g}, this implies
    \begin{equation}\label{eq-X<V}
        \alpha_{h_i^{-1}}((X\setminus N)\cap W_i)\subset  \bigcup_{g\in F}V_i\cap O_g\text{ for all }i=1,2,3.
    \end{equation}
    
    Analogously, since the sets $gA_i\subset \Gamma$ for $g\in F$ and $i=1,2,3$ are pairwise disjoint, at most one of them can have probability measure $>\frac 1 2$.
    Since we chose $\mu$ and $V_i$ such that $\mu(\alpha_g(x))(gA_i)>\frac 1 2$ for all $g\in F, x\in V_i\cap O_g$, and $i=1,2,3$, the sets 
    \begin{equation}\label{eq-V<U}
        \alpha_g( V_i\cap O_g)\subset U_1\text{ for }g\in F\text{ and }i=1,2,3
    \end{equation}
    are pairwise disjoint. 
    The combination of \eqref{eq-X<V} and \eqref{eq-V<U} implies $X\setminus N\prec U_1$ and finishes the proof. 
    Since $X\subset (X\setminus N) \cup \overline{N}$, while $\overline{N}\prec U_{2}$ and $U_{1}, U_{2}$ are disjoint, the two witnessing families combine to get $X\prec U$. 
\end{proof}

Now we are ready to prove the main theorem of the paper. 

\begin{theorem}\label{thm-mainthm-dr}
    Let $T\colon X\to X$ be a minimal surjective local homeomorphism of a compact metrizable space of finite covering dimension. 
    Then $G_T$ satisfies dynamical comparison. 
\end{theorem} 

\begin{proof}
    By \autoref{cor-reduction}, we may assume that $m_x\coloneqq |T^{-1}(x)|\geq 2$ for all $x\in X$. 
    In particular, $X$ is infinite and has no isolated points.
    By compactness, we moreover have $m_x<\infty$ for all $x\in X$. 
    Using that $T$ is a local homeomorphism, for each $x\in X$, there exist an open neighborhood $V_x$ of $x$ and pairwise disjoint open sets $U_1^x,\ldots,U_{m_x}^x$ such that 
    \[ T^{-1}(V_{x}) = \bigsqcup_{i = 1}^{m_{x}} U_{i}^{x}, \] 
    and $T$ restricts to a homeomorphism $T|_{U_i^x}\colon U_i^x\xrightarrow{\cong} V_x$ for each $i=1,\ldots,m_x$.
    Since $G_T$ has the thin boundary property by \autoref{thm-Lebesgue-thin}, we may find an open set $W_{x}$ such that $x\in W_x\subset \overline{W_x}\subset V_x$ and $\partial W_x$ is $G_T$-thin. 

    By compactness, we can choose a finite open cover $X=W_{x_1}\cup\ldots\cup W_{x_n}$ for some $x_1,\ldots,x_n\in X$.
    We define
    \[E\coloneqq \partial W_{x_1}\cup\ldots\cup\partial W_{x_n},\]
    which is a closed $G_T$-thin set by Lemma~\ref{lem-thin} \ref{item-thin-1}. 
    Define recursively,
    $P_1=W_{x_1}\setminus E$, and $P_j=W_{x_j}\setminus (E\cup {W_{x_1}}\cup\ldots\cup {W_{x_{j-1}}})$ for $j=2,\ldots,n$. 
    Since $E$ contains the boundary of each $W_{x_{j}}$, each $W_{x_{j}}\setminus E$ is clopen in $X\setminus E$. 
    Hence each $P_j$ is an open set with $P_j\subset W_{x_j}\subset V_{x_j}$, and
    \[X\setminus E=P_1\sqcup\ldots \sqcup P_n. \]
    For $j=1,\ldots,n$ and $i=1,\ldots,m_{x_j}$, we let
    \[U_{i,j}\coloneqq (T|_{U_i^{x_j}})^{-1}(P_j), \]
    and note that each $U_{i,j}\subset U_{i}^{x_j}$ is an open set with $T|_{U_{i,j}}\colon U_{i,j}\to P_j$ a homeomorphism. 
    Moreover, note that $\{U_{i,j}\mid j=1,\ldots,n, \ i=1,\ldots, m_{x_j}\}$ is a collection of pairwise disjoint sets. 
    Define 
    \[A_1\coloneqq \bigsqcup_{j=1}^{n}U_{1,j} \qquad \text{and} \qquad A_2\coloneqq \bigsqcup_{j=1}^{n}U_{2,j}.\]
    Then $A_1$ and $A_2$ are disjoint open sets, and $T|_{A_i}\colon A_i\to X\setminus E$ for $i=1,2$ is a homeomorphism.
    Rename the remaining sets $\{U_{i,j}\colon j=1,\ldots,n, \ i=3,\ldots,m_{x_j}\}$ as $A_3,\ldots,A_d$. 
    By construction, the sets $A_1,\ldots, A_d$ are pairwise disjoint and $T|_{A_i}\colon A_i\to T(A_i)$ is a homeomorphism for each $i=1,\ldots,d$.
    We define a semi-saturated orthogonal partial action $F_d\overset{\alpha}\curvearrowright X$ of the free group on $d$ generators by mapping the generators $a_i$ for $i=1,\dotsc, d$ to $T|_{A_i}\colon A_i\to T(A_i)$.
    To finish the proof, we verify the hypotheses of \autoref{thm-large-hyperbolic-subgroupoid}. 

    To see that $G_T$ contains $F_d\ltimes X$ as an open subgroupoid, denote by 
    \[S\colon A_1\sqcup \dotsb \sqcup A_d\to X\setminus E\] the restriction of $T$. 
    It follows from \cite[Theorem~9.6]{ExelSteinberg} (see also \cite[Theorem~4.1]{deCastroEunJi}) that $F_d\ltimes X\cong G_S$.
    It is easy to see from the definition of the topology on $G_T$ in terms of the basic open sets defined in \eqref{eq-basis-dr} that the natural map $G_S\to G_T$ given by $(x,n,y)\mapsto (x,n,y)$ identifies $G_S$ with an open subgroupoid of $G_T$. 
    
    To check the second condition of \autoref{thm-large-hyperbolic-subgroupoid}, note that
    \[T^{-1}(X\setminus E)=\bigsqcup_{j=1}^{n}T^{-1}(P_j)=\bigsqcup_{j=1}^{n}\bigsqcup_{i=1}^{m_{x_j}}U_{i,j}=A_1\sqcup A_2\sqcup\ldots\sqcup A_d=\dom(S).\]
    Now let $B_T(U,V,n,m)\subset G_T$ be a basic open subset as defined in \eqref{eq-basis-dr}, so that $T^n|_ U$ and $T^m|_{ V}$ are injective.
    By possibly shrinking $U$ and $V$ we may assume that $T^n|_{\overline U}$ and $T^m|_{\overline V}$ are injective as well. 
    If an element $(x, n-m, y)$ in $\overline{B_T(U,V,n,m)}$ is not in $G_{S}$, then $T^{k}(x)\in E$ for some $1\leq k\leq n$, or $T^{\ell}(y)\in E$ for some $1\leq \ell\leq m$. 
    Since $T^{n}(x) = T^{m}(y)$, it follows that 
    \[r\left(\overline{B_T(U,V,n,m)}\setminus G_S\right)\subset \bigcup_{k=1}^{n} T^{-k}(E) \cup \bigcup_{\ell = 1}^{m} T^{-n} T^{m-\ell}(E), \]
    with an empty union when $n=0$ or $m=0$. 
    The left-hand side is thus $G_{T}$-thin by Lemma~\ref{lem-thin} \ref{item-thin-1} as it is a closed subset of a finite union of $G_{T}$-thin sets. 
    Moreover, $B_T(U,V,n,m)$ only intersects finitely many bisections $\{g\}\times D_{g^{-1}}$ for $g\in F_d$ since this forces $|g|\leq n+m$. 

    By construction, the third condition of \autoref{thm-large-hyperbolic-subgroupoid} is satisfied for the generators $a_1,a_2\in F_d$.
    Finally, $F_d\ltimes X$ is amenable by \cite[Corollary~9.7]{ExelSteinberg}. 
    Thus, \autoref{thm-large-hyperbolic-subgroupoid} is applicable.
\end{proof}

\section{Thin boundaries from finite covering dimension}\label{sec-dimension}
This section is devoted to the proof of \autoref{thm-Lebesgue-thin}. 
We follow the ideas in \cite[Sections 3--4]{Szabo15}, \cite[Lemma~7.5]{KerrSzabo}, \cite[Proposition~3.10]{Buck}, and \cite[Section~3]{Lindenstrauss} developed for group actions.

The following definition is a substitute for the notion of \emph{topologically small sets} from group actions (see \cite{Buck,Szabo15}) which is sufficient for our applications to \'etale groupoids. 
Besides merely generalizing the theory to groupoids, our definition has the key advantage of not incorporating any implicit freeness assumption.
This is the reason why our techniques also prove the classical small boundary property for possibly non-free minimal group actions on finite-dimensional spaces (see \autoref{cor-sbp-group}).

\begin{defi}
For a locally compact Hausdorff {\'e}tale groupoid $G$ with a compact Hausdorff unit space $X$, we let $\mathcal{C}_{-1}=\{\emptyset\}$ and define recursively $\mathcal{C}_m$, for $m\geq 0$, to be the collection of all compact subsets $K\subset X$ which satisfy the following property: whenever $B_0,B_1\subset G$ are open bisections with associated homeomorphisms $\theta_{B_i}\colon s(B_i)\to r(B_i)$ (see \autoref{notation-theta}) and $F_i\subset s(B_i)$, for $i=0,1$, are compact disjoint sets, then the \textit{collision set} 
    \[\theta_{B_0}(K\cap F_0)\cap \theta_{B_1}(K\cap F_1)\] 
    belongs to $\mathcal{C}_{m-1}$.
\end{defi}
    Intuitively, $K\in\mathcal{C}_m$ if whenever two disjoint pieces of $K$ are moved to the same place, then the resulting collision set has one rank lower.

    The statement of the following lemma should be understood as a groupoid analogue of \cite[Lemma~7.5]{KerrSzabo} (see also \cite[Proposition~3.10]{Buck}).
\begin{lemma}\label{lem:finiteDimThin}
    Let $G$ be a minimal locally compact Hausdorff {\'e}tale groupoid with an infinite compact metrizable unit space $X$ of finite Lebesgue covering dimension. Let $K\subset X$ be a closed set and assume that $K\in\mathcal{C}_m$ for some $m\geq -1$.
    Then $K$ is $G$-thin.
\end{lemma}

\begin{proof}
    We prove the lemma by induction on $m$. The case $m=-1$ is trivial.
    Let $m\geq 0$ and assume that the lemma holds for $\mathcal{C}_{m-1}$. Let $W\subset X$ be a non-empty open set.
    Since $G$ is minimal and $X$ is infinite, it follows that $X$ has no isolated points. We can choose $d+2$ non-empty disjoint open sets
    \[R, W_1,\ldots, W_{d+1}\subset W, \]
    where $d\coloneqq \dim(X)$.

    We summarize the main ideas before proceeding with the proof.
    Using minimality, we cover $K$ by finitely many open sets which can be moved along bisections to (possibly overlapping) open subsets of each of the $W_i$.
    After refinement, we can color our open cover into $d+1$ colors so that sets of the same color do not overlap.
    The overlaps that arise after moving the sets of the $i$-th color into $W_i$ then belong to $\mathcal C_{m-1}$ and can be moved into $R$ by the inductive assumption.

    We proceed with the proof.
    By minimality of $G$ and compactness of $K$, for each $j=1,\ldots, d+1$, there exist finitely many open bisections $B_{j,1},\ldots,B_{j,m_j}\subset G$ such that 
    \[K\subset \bigcup_{\ell=1}^{m_j} s(B_{j,\ell}) \qquad \text{and} \qquad r(B_{j,\ell})\subset W_j, \ \text{ for all } \ell= 1,\ldots,m_j.\]
    For each $j=1,\ldots,d+1$, we obtain an open cover
    \[\mathcal{V}_j=\{s(B_{j,\ell}): \ell=1,\ldots,m_j\} \]
    of $K$. Let $\mathcal{V}=\bigvee_{j=1}^{d+1} \mathcal{V}_j$ be the common refinement of these open covers. Then $\mathcal{V}$ is a finite open cover of $K$ with the following property:
    for each $U\in\mathcal{V}$ and every $j=1,\ldots,d+1$, there exists an open bisection $B_{j,U}\subset G$ such that $U\subset s(B_{j,U})$ and $r(B_{j,U})\subset W_j$.

    Since $\dim(X)=d$, by \cite[Lemma~3.2]{Blanchardkirchberg}, we may refine $\mathcal{V}$ to a finite open cover
    \[\mathcal{P}=\mathcal{P}_1\sqcup \ldots \sqcup\mathcal{P}_{d+1} \]
    of $K$ such that
    \begin{enumerate}
        \item\label{item-P-1} for each $P\in\mathcal{P}$, there exists some $V\in \mathcal{V}$ such that $\overline{P}\subset V$, and
        \item\label{item-P-2} the closures of distinct members of $\mathcal{P}_k$ are disjoint, for $k=1,\ldots,d+1$.
    \end{enumerate}
    We say that two elements of $\mathcal{P}$ \textit{are of the same color} if they both belong to $\mathcal{P}_k$ for some $k\in\{1,\ldots,d+1\}$.
    Let $P\in \mathcal{P}$. Let $k\in\{1,\ldots,d+1\}$ be the unique element such that $P\in \mathcal{P}_k$. Using \ref{item-P-1}, there exists $V_P\in\mathcal{V}$ such that 
    $\overline{P}\subset V_P$. Set 
    \[B_P\coloneqq B_{k,V_P}.\] 
    Thus,
    $\overline{P}\subset s(B_{P})$ and $r(B_P)\subset W_k$.
    Moreover, put
    \[K_P\coloneqq K\cap \overline{P}. \]
    Note that $\{K_P\colon P\in\mathcal{P}\}$ is a collection of compact sets covering $K$, and $K_P\subset s(B_P)$ for all $P\in\mathcal{P}$.

    Next, let $\mathcal{F}$ be the following finite collection of ordered pairs:
    \[\mathcal{F}\coloneqq \{(P,Q)\mid P,Q\in\mathcal{P} \text{ are distinct sets of the same color}\} \]
    Choose pairwise disjoint non-empty open sets
    \[R_{(P,Q)}\subset R, \qquad (P,Q)\in \mathcal{F}.\]
    Note that for each element $(P,Q)\in\mathcal{F}$, we have by \ref{item-P-2}, that $\overline{P}\cap \overline{Q}=\emptyset$, and therefore we also have that $K_P\cap K_Q=\emptyset$.
    Since $K\in\mathcal{C}_m$, it follows that
    \[\theta_{B_P}(K_P)\cap \theta_{B_Q}(K_Q)\in \mathcal{C}_{m-1}.\]
    By the inductive assumption, this intersection is $G$-thin, 
    so we can find an open neighbourhood $N_{(P,Q)}\supset \theta_{B_P}(K_P)\cap \theta_{B_Q}(K_Q)$ such that $\overline{N_{(P,Q)}}\subset r(B_P)\cap r(B_Q)$ and
    \[\overline{N_{(P,Q)}}\prec_{G} R_{(P,Q)}.\]
    By composing the bisections $B_P$ with the bisections witnessing the subequivalence $\overline{N_{(P,Q)}}\prec_{G} R_{(P,Q)}$, we obtain
    \begin{equation}\label{eq:subequiv1}
        \bigcup_{(P,Q)\in \mathcal{F}}\theta_{B_P}^{-1}\Big(\overline{N_{(P,Q)}}\Big) \prec_G R.
    \end{equation}

    Next, for $P\in\mathcal{P}$, consider the closed set
    \[L_P\coloneqq K_P\setminus \bigcup_{\{Q \mid (P,Q)\in\mathcal{F}\}} \theta_{B_P}^{-1}(N_{(P,Q)}).\]
    We claim that $\{\theta_{B_P}(L_P)\mid P\in \mathcal{P}\}$ are pairwise disjoint.
    Indeed, let $P,Q\in \mathcal{P}$ be distinct elements.
    If $P$ and $Q$ are of different colors, say $P\in\mathcal{P}_k$ and $Q\in \mathcal{P}_{k'}$ for $k\neq k'$, then 
    \[\theta_{B_P}(L_P)\cap \theta_{B_Q}(L_Q)\subset r(B_P)\cap r(B_Q)\subset W_{k}\cap W_{k'}=\emptyset. \]
    If, on the other hand, $P$ and $Q$ are of the same color, and $y\in \theta_{B_P}(L_P)\cap \theta_{B_Q}(L_Q)$, choose $x_1\in L_P$ such that $\theta_{B_P}(x_1)=y$.
    Then 
    \[y\in \theta_{B_P}(K_P)\cap \theta_{B_Q}(K_Q)\subset N_{(P,Q)}.\]
    This implies that $x_1= \theta_{B_P}^{-1}(y)\in \theta_{B_P}^{-1}(N_{(P,Q)})$, contradicting the definition of $L_P$.
    We conclude that $\{\theta_{B_P}(L_P)\mid P\in \mathcal{P}\}$ are pairwise disjoint closed subsets contained in $W_1\cup\ldots\cup W_{d+1}$. It is straightforward to check that this implies that
    \[\bigcup_{P\in\mathcal{P}} L_P\prec_G W_1\cup\ldots\cup W_{d+1}.\]
    Combining with Equation~\eqref{eq:subequiv1}, we obtain that
    \[K\subset \left(\bigcup_{P\in\mathcal{P}} L_P\right)\cup \left(\bigcup_{(P,Q)\in \mathcal{F}}\theta_{B_P}^{-1}\big(\overline{N_{(P,Q)}}\big) \right)\prec_G W_1\cup \ldots\cup W_{d+1}\cup R\subset W, \]
    completing the proof.
\end{proof}

\begin{defi}\label{def:controlledBis}
    Let $G$ be a locally compact Hausdorff {\'e}tale groupoid with a compact Hausdorff unit space $X$. A \textit{controlled bisection} $\mathfrak{b}=(B,F)$ is a pair consisting of an open bisection $B\subset G$ and a closed set $F\subset X$ such that $F\subset s(B)$. 
    
    Let $\theta_B\colon s(B)\to r(B)$ denote the homeomorphism associated with the bisection $B$ (see \autoref{notation-theta}). 
    We write $\theta_\mathfrak{b}$ for the restriction $\theta_B|_F$.
    Moreover, for a subset $Y\subset X$, we write 
    \[\theta_\mathfrak{b}(Y)\coloneqq \theta_B(Y\cap F).\]

    Finally, a finite set $\mathcal{L}=\{\mathfrak{b}_1,\ldots,\mathfrak{b}_n\}$ of controlled bisections is called \textit{separated} if the corresponding closed sets $F_1,\ldots,F_n$ are pairwise disjoint.
\end{defi}

\begin{defi}\label{defi:generalPosition}
    Let $G$ be a locally compact Hausdorff {\'e}tale groupoid with a compact metrizable unit space $X$ of finite covering dimension $d=\dim(X)$. For a finite set $\mathcal{S}$ of controlled bisections, we say that a set $E\subset X$ is in \textit{$\mathcal{S}$-general position} if, for every non-empty separated subset $\mathcal{L}\subset \mathcal{S}$ we have
    \[\dim\left(\bigcap_{\mathfrak{b}\in\mathcal{L}}\theta_\mathfrak{b}(E)\right)\leq \max\{d-|\mathcal{L}|, -1\},\]
    with the convention that $\dim(\emptyset)=-1$.
\end{defi}
Intuitively, $E\subset X$ should be thought of as having positive codimension and its translates $\theta_\mathfrak b(E)$ under ``sufficiently different'' controlled bisections $\mathfrak b=(B,F)$ should be thought of as being in ``general position'' in the sense that any further intersection with them reduces the dimension by one. 
Here, the precise meaning of ``sufficiently different'' is implemented by our definition of \emph{separated subsets} and mimics the role of distinct group elements in the definition of \emph{topologically small sets} in Szab\'o's \cite[Definition 3.1]{Szabo15} (see also \cite[Definition 3.2]{Buck}).

The next lemma is a direct groupoid analogue of \cite[Lemma~3.6]{Szabo15}.
Note that, since our notion of separated subsets plays the role of freeness in the proof of \cite[Lemma~3.6]{Szabo15}, we do not have any principality assumptions on the groupoid.

\begin{lemma}\label{lem:bdryGeneralPosition}
    Let $G$ be a locally compact Hausdorff {\'e}tale groupoid with a compact metrizable unit space $X$ of finite covering dimension $d=\dim(X)$.
    Let $\mathcal{S}$ be a finite set of controlled bisections. Let $A\subset X$ be a closed set and let $O\subset X$ be an open set, such that $A\subset O$.
    Then there exists an open set $W\subset X$ such that 
    \[A\subset W\subset \overline{W}\subset O,\]
    and $\partial W$ is in $\mathcal{S}$-general position.
\end{lemma}

\begin{proof}
    Choose an open set $W_0\subset X$ such that $A\subset W_0\subset\overline{W_0}\subset O$. 
    Write $\mathcal{S}=\{\mathfrak{b}_1,\ldots, \mathfrak{b}_N\}$, with $\mathfrak{b}_i=(B_i,F_i)$ for every $i=1,\ldots, N$. Let $x\in \partial W_0\subset O\setminus A$. Consider the collection of ordered pairs
    \[\mathcal{I}\coloneqq \{(i,j)\in \{1,\ldots,N\}^2 \mid F_i\cap F_j=\emptyset\},\]
    and define for each $(i,j)\in\mathcal{I}$ the set
    \[H_{(i,j)}\coloneqq \begin{cases}
        F_i, & \text{if } x\notin F_i,\\
        F_j, & \text{if } x\in F_i.
    \end{cases} \]
    Note that  
    \[\widehat C_x\coloneqq (O\setminus A)\cap\bigcap_{(i,j)\in\mathcal{I}}X\setminus H_{(i,j)}\] is an open neighbourhood of $x$.
    Choose an open set $C_x\subset X$ such that
    \[x\in C_x\subset \overline{C_x}\subset \widehat{C}_x. \]
    By construction, if $F_i\cap F_j=\emptyset$ for some $i,j\in\{1,\ldots,N\}$, then 
    \begin{equation}\label{eq:coverAdvoid}
        \widehat{C}_x\cap F_i=\emptyset \qquad \text{or} \qquad \widehat{C}_x\cap F_j=\emptyset.
    \end{equation}
    By compactness of $\partial W_0$, there exists a finite subcover
    \[\partial W_0\subset C_{x_1}\cup \ldots\cup C_{x_M} \]
    for some $x_1,\dotsc,x_M\in \partial W_0$. 
    We abbreviate $C_i=C_{x_i}$ and $\widehat{C}_{i}=\widehat{C}_{x_i}$ and write $\mathcal{P}_i=C_0\cup C_1\cup\ldots\cup C_i$ with $C_0=\emptyset$ for all $i=0,\ldots,M$.
    We will construct inductively open sets
    \[W_0,W_1,\ldots,W_M\subset X \]
    with the following properties:
    \begin{enumerate}
        \item\label{item-W-1} $\overline{W}_i\subset W_0\cup \mathcal{P}_M$, for $i=0,\dotsc,M$.
        \item\label{item-W-2} $\partial W_i\cap \mathcal{P}_i$ is in $\mathcal{S}$-general position, for $i=0,\ldots,M$.
    \end{enumerate}
    Clearly, $W_0$ satisfies the desired conditions. We assume now that we have constructed $W_i$ with the required properties for some $0\leq i<M$, and we will construct $W_{i+1}$.
    Let $\mathcal{L}\subset \mathcal{S}$ be a non-empty separated collection. 
    By the inductive hypothesis,
    \[\dim\left(\bigcap_{\mathfrak{b}\in\mathcal{L}}\theta_\mathfrak{b}(\partial W_i\cap \mathcal{P}_i)\right)\leq \max\{d-|\mathcal{L}|, -1\}. \]
    As can be found in \cite{Engelking} (see also stated in \cite[Section~3]{Szabo15} and in \cite[Section~3]{Lindenstrauss}), we can find an $F_\sigma$-set\footnote{Recall that an $F_\sigma$-set is a countable union of closed sets.}
    \[E_{\mathcal{L},i}\subset \bigcap_{\mathfrak{b}\in\mathcal{L}}\theta_\mathfrak{b}\left(\partial W_i\cap \mathcal{P}_i\right)\] 
    such that $\dim(E_{\mathcal{L},i})\leq 0$ and
    \begin{equation}\label{eq:EdimReduc}
        \dim\left(\bigcap_{\mathfrak{b}\in\mathcal{L}}\theta_\mathfrak{b}(\partial W_i\cap \mathcal{P}_i)\setminus E_{\mathcal{L},i}\right)\leq
        \max\left\{\dim\left(\bigcap_{\mathfrak{b}\in\mathcal{L}}\theta_\mathfrak{b}(\partial W_i\cap \mathcal{P}_i)\right)-1, -1\right\}
    \end{equation}
    By the same reasoning, 
    there exists a zero-dimensional $F_\sigma$-set $Z\subset X$ such that
    \begin{equation}\label{eq:XZ}
        \dim(X\setminus Z)=d-1.
    \end{equation}
    Set
    \begin{equation}\label{eq-def-Ei}
        E_i\coloneqq Z\cup \bigcup_{\substack{\emptyset\neq \mathcal{L}\subset \mathcal{S} \\ \text{separated}}} \bigcup_{\mathfrak{b}=(B,F)\in\mathcal{S}} F\cap (\theta_{B}^{-1}(E_{\mathcal{L},i}\cap r(B))). 
    \end{equation}
    Note that $E_i$ is a zero-dimensional $F_\sigma$-set as a finite union of zero dimensional $F_\sigma$-sets.
    Using the inductive hypothesis, we have 
    \[\overline{W_i\cap C_{i+1}}\subset (W_0\cup \mathcal{P}_M)\cap \widehat{C}_{i+1}. \]
    Hence, applying \cite[Lemma~3.3]{Szabo15}, we can find an open set $U_i\subset X$ such that
    \begin{equation}\label{eq-def-Ui}
        \overline{W_i\cap C_{i+1}}\subset U_i\subset \overline{U}_i\subset (W_0\cup \mathcal{P}_M)\cap \widehat{C}_{i+1}  
    \end{equation}
    and 
    \begin{equation}\label{eq-def-dU}
        \partial U_i\cap E_i=\emptyset.
    \end{equation}
    Define
    $W_{i+1}\coloneqq W_i\cup U_i$. 
    We check that conditions \ref{item-W-1} and \ref{item-W-2} are satisfied for $W_{i+1}$.
    Condition \ref{item-W-1} follows from
    \[\overline{W}_{i+1}=\overline{W}_i\cup \overline{U}_i\subset (W_0\cup \mathcal{P}_M) \cup \overline{U}_i= W_0\cup \mathcal{P}_M.\]

    We complete the induction by proving condition \ref{item-W-2} which states that
    $\partial W_{i+1}\cap \mathcal{P}_{i+1}$
    is in $\mathcal{S}$-general position.
    Fix a non-empty separated collection $\mathcal{L}=\{\mathfrak{b}_1,\ldots,\mathfrak{b}_r\}\subset \mathcal{S}$.

    Using \eqref{eq-def-Ui} and thus $(\partial W_i\setminus U_i)\cap C_{i+1}\subset \overline{W_i\cap C_{i+1}}\setminus U_i=\emptyset$ at the last step, we get 
    \begin{equation*}\label{eq:key}
        \begin{aligned}
        \partial W_{i+1}\cap \mathcal P_{i+1}&= \Big (\Big(\overline{W_i}\cup \overline{U_i}\Big)\setminus (W_i\cup U_i)\Big)\cap (\mathcal P_{i}\cup C_{i+1})\\
        &= \big(({\partial W_i}\setminus U_i)\cup ({\partial U_i}\setminus W_i)\big)\cap (\mathcal P_i\cup C_{i+1})\\
        &\subset (\partial W_i\cap \mathcal P_i)\cup \partial U_i,
        \end{aligned}
    \end{equation*}

    and thus for each $\mathfrak b\in \mathcal L$ the inclusion
    \begin{equation}\label{eq-bWP}
        \theta_\mathfrak{b}(\partial W_{i+1}\cap \mathcal{P}_{i+1})\subset \theta_\mathfrak{b}(\partial W_i\cap \mathcal{P}_i)\cup \theta_\mathfrak{b}(\partial U_i). 
    \end{equation}
    Using Equation~\eqref{eq:coverAdvoid} and the fact that $\mathcal{L}$ is a separated collection, we have that $\theta_{\mathfrak{b}_j}(\partial U_i)\cap \theta_{\mathfrak{b}_k}(\partial U_i)=\emptyset$ for all distinct $j,k\in\{1,\ldots,r\}$. Using this observation 
    in the second step, we have
    \begin{align*}
        \bigcap_{\mathfrak{b}\in\mathcal{L}}\theta_\mathfrak{b}(\partial W_{i+1}\cap \mathcal{P}_{i+1})&\subset \bigcap_{\mathfrak{b}\in\mathcal{L}}\big(\theta_\mathfrak{b}(\partial W_i\cap \mathcal{P}_i)\cup \theta_\mathfrak{b}(\partial U_i)\big)\\ 
        &\subset \left(\bigcap_{\mathfrak{b}\in\mathcal{L}} \theta_{\mathfrak b}(\partial W_i\cap \mathcal{P}_i) \right)\cup \bigcup_{k=1}^{r} \left(\theta_{\mathfrak{b}_k}(\partial U_i)\cap\bigcap_{\mathfrak{b}\in\mathcal{L}\setminus \{\mathfrak b_k\}} \theta_{\mathfrak{b}}(\partial W_i\cap \mathcal{P}_i)\right),
    \end{align*}
   where the empty intersection (in case that $\mathcal L=\{\mathfrak b_k\}$) is understood as $X$. 
   In order to complete the induction, we have to bound the dimension of the latter set by $\max\{d-r,-1\}$. 
   Since the set on the right hand side is a finite union of $F_\sigma$-sets, it suffices to bound the dimension of each of the individual sets separately. 

   By the inductive assumption, we have
   \[\dim\left(\bigcap_{\mathfrak{b}\in\mathcal{L}} \theta_{\mathfrak b}(\partial W_i\cap \mathcal{P}_i)\right)\leq \max\{d-r, -1\}. \]

    To bound the dimension of the set $\theta_{\mathfrak{b}_k}(\partial U_i)\cap\bigcap_{\mathfrak{b}\in\mathcal{L}\setminus \{\mathfrak b_k\}} \theta_{\mathfrak{b}}(\partial W_i\cap \mathcal{P}_i)$ for $k=1,\dotsc,r$, we treat the cases $r=1$ and $r\geq 2$ separately. 

    If $r=1$, then $\dim(\theta_{\mathfrak{b}_k}(\partial U_i))=\dim(F_k\cap \partial U_i)\leq \dim(X\setminus Z)\leq d-1$ since 
    $ \partial U_i\subset X\setminus E_i\subset X\setminus Z$ by \eqref{eq:XZ}, \eqref{eq-def-Ei}, and \eqref{eq-def-dU}. 

    If $r\geq 2$, we note that $\theta_{\mathfrak b_k} (\partial U_i)\cap E_{\mathcal L\setminus\{\mathfrak b_k\},i}=\emptyset$ by \eqref{eq-def-dU} and \eqref{eq-def-Ei}.
    Using this at the first step, \eqref{eq:EdimReduc} at the second step and the inductive hypothesis at the last step, we obtain 
    \begin{equation}
        \begin{aligned}
        &\dim\left(\theta_{\mathfrak{b}_k}(\partial U_i)\cap\bigcap_{\mathfrak{b}\in\mathcal{L}\setminus \{\mathfrak b_k\}} \theta_{\mathfrak{b}}(\partial W_i\cap \mathcal{P}_i)\right)\\
        &\leq\dim\left(\bigcap_{\mathfrak{b}\in\mathcal{L}\setminus\{\mathfrak{b}_k\}}\theta_\mathfrak{b}(\partial W_i\cap \mathcal{P}_i)\setminus E_{\mathcal{L}\setminus\{\mathfrak{b}_k\},i}\right)\\
        &\leq \max\left\{\dim\left(\bigcap_{\mathfrak{b}\in\mathcal{L}\setminus\{\mathfrak{b}_k\}}\theta_\mathfrak{b}(\partial W_i\cap \mathcal{P}_i)\right)-1,-1\right\}\\
        &\leq \max\{d-(r-1)-1,-1\}\\
        &= \max \{d-r,-1\}.
        \end{aligned}
    \end{equation}
    This completes the induction. 
    
    Finally, we set $W\coloneqq W_M$ and claim that $W$ satisfies the required properties. 
    First, we have 
    \[A\subset W_0\subset W_1\subset\ldots\subset W_M=W.\]
    Second, by the inductive hypothesis, 
    \begin{equation}\label{eq:induct1}
        \overline{W}_M\subset W_0\cup \mathcal{P}_M=W_0\cup C_0\cup \ldots \cup C_M\subset O. 
    \end{equation}
    In particular, we have $\partial W_M=\partial W_M\cap \mathcal P_M$ since $W_M$ is open and $W_0\subset W_M$. 
    Finally, by the inductive hypothesis $\partial W_M=\partial W_M\cap \mathcal{P}_M$ is in $\mathcal{S}$-general position. 
\end{proof}
Before stating the next lemma, observe that whenever $X$ is a second countable compact Hausdorff space then there exists a countable family $\mathcal{K}$ of compact subsets of $X$ with the following property: whenever $A\subset X$ is a compact set and $O\subset X$ is an open set such that $A\subset O$, then there exists $K\in\mathcal{K}$ such that
\[A\subset \mathrm{int}(K)\subset K\subset O. \]
Indeed, let $\{B_1,B_2,\ldots \}$ be a countable basis for $X$, then
\[\mathcal{K}\coloneqq \{\overline{B}_{i_1}\cup\ldots\cup\overline{B}_{i_n}: n\geq 1,\ i_1,\ldots,i_n\in\N\} \]
satisfies the required properties.

The next lemma provides the intermediate step between general position and thinness. 
It converts the separated intersection bound into finite collision rank, to which Lemma~\ref{lem:finiteDimThin} applies. 
In comparison, for group actions, the passage from topological smallness to thinness is established directly in \cite[Proposition 3.10]{Buck} and \cite[Lemma 7.5]{KerrSzabo}. 

\begin{lemma}\label{lem:ctrl_disjnt_fin_rank}
    Let $G$ be a second countable locally compact Hausdorff {\'e}tale groupoid with a compact metrizable unit space $X$. Choose a countable basis $\mathcal{B}$ of open bisections of $G$ and choose a countable family $\mathcal{K}$ of compact subsets of $X$ with the following property:
    whenever $A\subset X$ is a compact set and $O\subset X$ is an open set such that $A\subset O$, then there exists $K\in\mathcal{K}$ such that
    \[A\subset \mathrm{int}(K)\subset K\subset O. \]
    Let $m\in \N_0$ and suppose that $E\subset X$ is a closed set and satisfies that
    whenever $\mathcal{L}=\{\mathfrak{b}_1,\ldots, \mathfrak{b}_{m+1}\}$ is a separated collection of controlled bisections $\mathfrak{b}_i=(B_i,F_i)$ with $B_i\in\mathcal{B}$ and $F_i\in\mathcal{K}$ for $i=1,\ldots,m+1$, then one has $\bigcap_{\mathfrak{b}\in\mathcal{L}}\theta_\mathfrak{b}(E)=\emptyset$. 
    Then \[E\in \mathcal{C}_{m-1}.\]
\end{lemma}

\begin{proof}
    We prove the lemma by induction on $m$. Assume first that $m=0$. 
    Then for any $x\in X$, there exists an open bisection $B\in\mathcal{B}$ and a $K\in \mathcal K$ such that 
    \[x\in \mathrm{int}(K)\subset K \subset B\subset X.\]
    By assumption, we have 
    \[\{x\}\cap E\subset K\cap E=\theta_{(B,K)}(E)=\emptyset.\]
    Since $x$ was arbitrary, we conclude that $E=\emptyset \in \mathcal C_{-1}$. 

    Assume now that $m\geq 1$. We want to show that $E\in\mathcal{C}_{m-1}$.
    Let $B_0,B_1\subset G$ be open bisections and let $F_0\subset s(B_0)$ and $F_1\subset s(B_1)$ be disjoint compact sets. We must show that
    \[R\coloneqq \theta_{B_0}(E\cap F_0)\cap \theta_{B_1}(E\cap F_1) \]
    belongs to $\mathcal{C}_{m-2}$. Suppose not. Then, by the inductive hypothesis, there is a separated collection $\mathcal{L}=\{\mathfrak{c}_1,\ldots,\mathfrak{c}_{m}\}$ of controlled bisections $\mathfrak{c}_i=(C_i,H_i)$ with $C_i\in\mathcal{B}$ and $H_i\in\mathcal{K}$ for $i=1,\ldots,m$ such that 
    \[ \bigcap_{\mathfrak{c}\in\mathcal{L}}\theta_\mathfrak{c}(R)\neq\emptyset.\]
    Choose $z\in \bigcap_{\mathfrak{c}\in\mathcal{L}}\theta_\mathfrak{c}(R)$, i.e., for each $j=1,\ldots,m$, there is $y_j\in R\cap H_j$ such that $\theta_{C_j}(y_j)=z$.
    Since $y_j\in R$, choose $x_j^{0}\in E\cap F_0$ and $x_j^1\in E\cap F_1$ such that
    \[\theta_{B_0}(x_j^0)=y_j=\theta_{B_1}(x_j^1).\]
    Consider the following $m+1$ points of $E$:
    \[x_1^{0}, x_2^{0},\ldots, x_m^{0}, x_{1}^1 \]
    and note that they are distinct. Consider the $m+1$ bisections obtained by the following compositions:
    \[C_1B_0, C_2B_0,\ldots, C_mB_0, C_1B_1. \]
    Clearly, $\theta_{C_iB_0}(x_i^{0})=z=\theta_{C_1B_1}(x_1^1)$ for $i=1,\ldots,m$.
    Since $\mathcal{B}$ is a basis of open bisections we can find open bisections $D_j^0\subset C_jB_0$ for $j=1,\ldots, m$ and $D_1^1\subset C_1B_1$ in $\mathcal{B}$ such that 
    \[x_j^0\in s(D_j^0), \ \text{ for } j=1,\ldots,m, \qquad \text{and} \qquad x_1^1\in s(D_1^1).\] 
    Now choose pairwise disjoint open neighbourhoods $x_j^0\in O_j^0\subset X$, for $j=1,\ldots, m$ and 
    $x_1^1\in O_1^1\subset X$ such that
    \[O_j^0\subset s(D_j^0), \ \text{ for } j=1,\ldots,m, \qquad \text{and} \qquad O_1^1\subset s(D_1^1). \]
    By the property of the collection $\mathcal{K}$, there exist compact sets $F_j^0$ for $j=1,\ldots,m$ and $F_1^1$ in $\mathcal{K}$ such that
    \[x_j^0\in F_j^0\subset O_j^0 \ \text{ for } j=1,\ldots, m, \qquad \text{and} \qquad x_1^1\in F_1^1\subset O_1^1. \]
    Consider now the separated collection of $m+1$ controlled bisections
    \[\mathfrak{d}_1=(D_1^0,F_1^0), \mathfrak{d}_2=(D_2^0,F_2^0),\ldots, \mathfrak{d}_m=(D_m^0,F_m^0), \mathfrak{d}_{m+1}=(D_1^1,F_1^1). \]
    By our assumption on $E$, one has
    \[\bigcap_{j=1}^{m+1}\theta_{\mathfrak{d}_j}(E)=\emptyset. \]
    This is a contradiction, as one can readily verify that $z$ belongs to this intersection.
\end{proof}

We need an auxiliary compactness fact, corresponding to \cite[Lemma~3.7]{Szabo15}. The proof is short and identical to the one appearing in \cite{Szabo15}, so
we omit the details.

\begin{prop}\label{prop:cptfact}
    Let $G$ be a locally compact Hausdorff {\'e}tale groupoid with a compact Hausdorff unit space $X$, let $\mathcal{S}$ be a finite set of controlled bisections, and let $n\in \N$. Suppose $E\subset X$ is a closed set and satisfies:
    for every separated collection $\mathcal{L}\subset \mathcal{S}$ with $|\mathcal{L}|=n$ one has 
    \[\bigcap_{\mathfrak{b}\in\mathcal{L}}\theta_\mathfrak{b}(E)=\emptyset.\]
    Then there is an open neighbourhood $V\supset E$ such that 
    \[\bigcap_{\mathfrak{b}\in\mathcal{L}}\theta_\mathfrak{b}\left(\overline{V}\right)=\emptyset,\]
    for every separated collection $\mathcal{L}\subset \mathcal{S}$ with $|\mathcal{L}|=n$.
\end{prop}

The next theorem is the groupoid analogue of the diagonal argument in \cite[Theorem~3.8]{Szabo15}.

\begin{theorem}\label{thm:findimThinBdry}
    Let $G$ be a second countable locally compact Hausdorff {\'e}tale groupoid with a compact metrizable unit space $X$ of finite Lebesgue covering dimension $d=\dim(X)$. Let $A\subset X$ be a closed set and let $O\subset X$ be an open set such that $A\subset O$. 
    Then there is an open set $W\subset X$ such that 
    \[A\subset W\subset \overline{W}\subset O,\]
    and $\partial W\in \mathcal{C}_{d-1}$.
\end{theorem}

\begin{proof}
    Let $\mathcal{B}$ be a countable basis of open bisections for $G$. 
    Let $\mathcal{K}$ be a countable family of compact subsets in $X$ with the property that whenever $C\subset X$ is a closed set and $L\subset X$ is an open set with $C\subset L$, then there is $K\in\mathcal{K}$ such that
    \[C\subset\mathrm{int}(K)\subset K\subset L.\]
    Enumerate the countable collection of all controlled bisections
    \[\mathfrak{b}_1=(B_1,F_1), \mathfrak{b}_2=(B_2,F_2),\ldots \]
    such that $B_i\in \mathcal{B}$ and $F_i\in\mathcal{K}$ for all $i\in\mathbb{N}$.
    Choose an open set $O_0\subset X$ such that 
    \[A\subset O_0\subset \overline{O}_0\subset O. \]
    We will construct inductively open sets $O_0\subset O_1\subset O_2\subset\ldots\subset O$ and open sets $\ldots N_2\subset N_1\subset X$ by alternately applying \autoref{lem:bdryGeneralPosition} and \autoref{prop:cptfact}, such that for each $k\in\N$ one has:
    \begin{enumerate}
        \item\label{item-ON1} $\partial O_k\subset N_k$, and
        \item\label{item-ON2} for every separated collection $\mathcal{L}\subset \{\mathfrak{b}_1,\ldots,\mathfrak{b}_k\}$ with $|\mathcal{L}|=d+1$, one has
        \[\bigcap_{\mathfrak{b}\in\mathcal{L}}\theta_\mathfrak{b}\big(\overline{N}_k\big)=\emptyset. \]
    \end{enumerate}
    For the construction of $O_1$, apply \autoref{lem:bdryGeneralPosition} to find an open set $O_1\subset X$ such that 
    \[\overline{O}_0\subset O_1\subset \overline{O}_1\subset O\]
    and $\partial O_1$ is in $\{\mathfrak{b}_1\}$-general position, i.e., $\dim(\theta_{\mathfrak{b}_1}(\partial O_1))\leq d-1$. 
    Hence, for $k=1$, condition~\ref{item-ON2} is satisfied for $\partial O_1$ (since, if $d=0$, this means that $\theta_{\mathfrak{b}_1}(\partial O_1)=\emptyset$ and if $d>0$, then the condition is vacuous).
    Applying \autoref{prop:cptfact} and using that $\partial O_1\subset O$, we can find an open set $N_1\subset X$
    such that
    \[\partial O_1\subset N_1\subset \overline{N}_1\subset O \]
    and condition~\ref{item-ON2} is satisfied for $\overline{N}_1$ (and $k=1$).

    Suppose now that $O_k,N_k$ have been constructed for some $k\geq 1$. 
    Apply \autoref{lem:bdryGeneralPosition} with respect to the inclusion $\overline{O}_k\subset O_k\cup N_k$ and the finite family $\mathcal{S}_{k+1}\coloneqq \{\mathfrak{b}_1,\ldots,\mathfrak{b}_{k+1}\}$ of controlled bisections, to find an open set $O_{k+1}\subset X$ such that
        \begin{equation}\label{eq:Ok}
        \overline{O}_k\subset O_{k+1}\subset \overline{O}_{k+1}\subset O_k\cup N_k,
    \end{equation}
    and $\partial O_{k+1}$ is in $\mathcal{S}_{k+1}$-general position, i.e., whenever $\mathcal{L}\subset \{\mathfrak{b}_1,\ldots, \mathfrak{b}_{k+1}\}$ is a separated collection of controlled bisections, then
    \[\dim\left(\bigcap_{\mathfrak{b}\in\mathcal{L}} \theta_\mathfrak{b}(\partial O_{k+1})\right)\leq \max\{d-|\mathcal{L}|, -1\}. \]
    Thus, for $k+1$, condition~\ref{item-ON2} is satisfied for $\partial O_{k+1}$. 
    By \autoref{prop:cptfact}, we can choose an open set $N_{k+1}\subset X$ such that $\partial O_{k+1}\subset N_{k+1}$ and condition~\ref{item-ON2} is satisfied for $\overline{N}_{k+1}$. Note that, using Equation~\eqref{eq:Ok}, we have
    \[\partial O_{k+1}= \overline{O}_{k+1}\setminus O_{k+1}\subset (O_k\cup N_k)\setminus O_{k+1}\subset N_k. \]
    Hence, by refining the choice of $N_{k+1}$, we may assume that
    \[\partial O_{k+1}\subset N_{k+1}\subset \overline{N}_{k+1}\subset N_k. \]
    The induction is now complete, and we note that, by construction, we additionally have the following properties:
    \begin{enumerate}\setcounter{enumi}{2}
        \item\label{item-ON3} $\overline{O}_{k+1}\subset O_k\cup N_k$, for all $k\in\N$, and
        \item\label{item-ON4} $\overline{N}_{k+1}\subset N_k$, for all $k\in \N$.
    \end{enumerate}
    Define $W=\bigcup_{k=0}^{\infty}O_k$. 
    Then $W$ is open and $A\subset O_0\subset W$ by construction.
    Next, we claim that for all $k\in \N$, we have 
    \begin{equation}\label{eq:inductInclusion}
         W\subset O_k\cup N_k.
    \end{equation}
    Indeed, we show that $O_l \subset O_k\cup N_k$ by induction on $l\geq k$. 
    The assertion for $l=k$ is trivial. 
    Assume the assertion for a fixed $l\geq k$. 
    Using condition \ref{item-ON3} at the first step, condition \ref{item-ON4} at the second step, and the inductive assumption at the last step, we obtain 
    \[O_{l+1}\subset O_l\cup N_l\subset O_l\cup N_k \subset O_k\cup N_k.\]

    Next, we claim that
    \begin{equation} \label{eq-partialWinN}
        \partial W\subset {N}_k.
    \end{equation}
    Indeed, for $x\in \partial W$, there exists $x_n\in W$ for $n\in \N$ such that $x_n\xrightarrow{n\to \infty}x$. 
    If $x_n\in N_{k+1}$ for infinitely many $n$, then in particular $x\in \overline{N_{k+1}}\subset N_k$. 
    If $x_n\notin N_{k+1}$ for infinitely many $n$, then for these $n$ we have $x_n\in O_{k+1}$ by \eqref{eq:inductInclusion} and thus $x\in \overline{O_{k+1}}\subset O_k\cup N_k$ by condition \ref{item-ON3}. 
    Since $x\in \partial W\subset X\setminus W \subset X\setminus O_k$, we get $x\in N_k$ as desired.
    
    As an immediate consequence of \eqref{eq-partialWinN}, we get $\overline W=W\cup \partial W\subset O\cup N_k=O$.
    Equation \eqref{eq-partialWinN} moreover implies that $\partial W$ satisfies condition \ref{item-ON2} for all $k\in \N$.
    It follows from \autoref{lem:ctrl_disjnt_fin_rank} that $\partial W\in \mathcal{C}_{d-1}$.
\end{proof}

\begin{cor}\label{cor-sbp-groupoid}
     Let $G$ be a minimal second countable locally compact Hausdorff {\'e}tale groupoid with a compact metrizable unit space $X$ of finite Lebesgue covering dimension. 
     Then $G$ has the thin boundary property. 
\end{cor}
\begin{proof}
    Without loss of generality, we may assume that $X$ is infinite. 
    Now the statement follows from the combination of \autoref{thm:findimThinBdry} and \autoref{lem:finiteDimThin}. 
\end{proof} 

Recall from \cite{Shub1991, Lindenstrauss} that an action $\Gamma\curvearrowright X$ of a discrete group on a compact space has the \emph{small boundary property} if there is a basis $\mathcal B$ for the topology on $X$ such that for every $B\in \mathcal B$ and every $\Gamma$-invariant Borel probability measure $\mu$ on $X$, we have $\mu(\partial B)=0$. 

The following corollary partially generalizes results from \cite{Szabo15,Gardella2024,Kerr2024} since it does not use any freeness assumption. This however comes at the cost of assuming minimality.
\begin{cor}\label{cor-sbp-group}
Let $\Gamma\curvearrowright X$ be a minimal action of a countable discrete group on a compact metrizable space of finite Lebesgue covering dimension. 
Then $\Gamma\curvearrowright X$ has the small boundary property.
\end{cor}
\begin{proof}
    Without loss of generality, we can assume that $X$ is infinite and thus has no isolated points. 
    It follows that for any $\Gamma\ltimes X$-thin set $A\subset X$, we have $\mu(A)=0$ for all $\Gamma$-invariant probability measures $\mu$ on $X$. 
    The corollary thus follows from \autoref{thm-Lebesgue-thin}.
\end{proof} 

\bibliography{references.bib} 
\bibliographystyle{alphaurl}

\end{document}